\documentclass[11pt,leqno]{amsart}
\makeatletter
\AtBeginDocument{%
  \let\original@@tocwrite=\@tocwrite
  \newif\ifAHVflag
  \def\jlreq@uniqtoken{\jlreq@uniqtoken}
  \def\jlreq@endmark{\jlreq@endmark}
  \long\def\jlreq@getfirsttoken#1#{\jlreq@getfirsttoken@#1\bgroup\jlreq@endmark}
  \long\def\jlreq@getfirsttoken@#1#2\jlreq@endmark#3\jlreq@endmark{#1}
  \renewcommand{\@tocwrite}[2]{%
    \begingroup
      \AHVflagfalse % true if section is not empty
      \@ifempty{#2}{}{%
        \expandafter\expandafter\expandafter\ifx\jlreq@getfirsttoken#2\jlreq@uniqtoken{}\jlreq@endmark\Sectionformat\expandafter\@firstoftwo\else\expandafter\@secondoftwo\fi
        {%
          \def\Sectionformat##1##2{\@ifempty{##1}{}{\AHVflagtrue}}%
          #2
        }{\AHVflagtrue}%
      }%
      \def\@tempa{}%
      \ifAHVflag\def\@tempa{\original@@tocwrite{#1}{#2}}\fi
    \expandafter\endgroup
    \@tempa
  }%
}
\makeatother

\usepackage{a4wide}
\usepackage{amssymb}
\usepackage{amsmath}
\usepackage{amsthm}
\usepackage{hyperref}
\usepackage[all]{xy}
\usepackage[usenames,dvipsnames]{xcolor}
\usepackage{lmodern}
\usepackage[T1]{fontenc}

\newcommand{\Hom}{\operatorname{Hom}}

\newcommand{\ind}{\operatorname{ind}}

\newcommand{\Ker}{\operatorname{Ker}}
\newcommand{\End}{\operatorname{End}}

\newcommand{\Gal}{\operatorname{Gal}}

\newcommand{\Mod}{\operatorname{Mod}}

\DeclareMathOperator{\vol}{vol}
\DeclareMathOperator{\Aut}{Aut}

 \DeclareMathOperator{\lcm}{lcm}

\theoremstyle{plain} 
\newtheorem{theorem}{Theorem}[section]
\newtheorem{corollary}[theorem]{Corollary}
\newtheorem{lemma}[theorem]{Lemma}
\newtheorem{proposition}[theorem]{Proposition}

\newtheorem{fact}[theorem]{Fact}

\newtheorem{proposition-definition}[theorem]{Proposition-Definition}

\theoremstyle{definition}
\newtheorem{definition}[theorem]{Definition}

\theoremstyle{remark}
\newtheorem{remark}[theorem]{Remark}

\numberwithin{equation}{section}

\title{Existence of  supersingular  representations of  $p$-adic reductive groups}

\author{M.-F. Vign\'eras}
\address[M.-F. Vign\'eras]{Institut de Math\'ematiques de Jussieu, 175 rue du Chevaleret, Paris 75013 France}
\email{vigneras@math.jussieu.fr}

\date{\today}

\begin{document} 

\maketitle
%\abstract   
\setcounter{tocdepth}{2}  
\tableofcontents

Let  $G$ be a connected   simple adjoint $p$-adic group not  isomorphic  to  a projective linear group $PGL_m(D)$  of a  division algebra $D$ $(m\geq 2$), or an adjoint ramified  unitary group $PU(h)$ of a split hermitian form $h$ in $3$ variables. 
We prove that  $G $ admits an irreducible admissible  supercuspidal (= supersingular) representation over any  field of characteristic $p$.

\section{Introduction} \label{sec:introduction}  
Throughout this paper,  
$F$ is a local non-archimedean field of characteristic $a$ and residue field  of characteristic $p$ with $q$ elements,  $G={\bf G}(F)$ where $\bf G$ is   a connected reductive $F$-group,  $C$   a field (of coefficients) of characteristic $c$ and $C^{alg}$ an algebraic closure of $C$.

\bigskip Recent applications of automorphic forms to number theory have imposed the study of smooth representations of $G $  on $C$-vector spaces for $C$  not algebraically closed  and often finite with $c=p$.  Indeed one expects a strong relation, \`a la Langlands, with $C$-representations of the Galois group of $F$ - the only established case, however, is that of $GL(2,\mathbb Q_p)$.

\bigskip An irreducible admissible $C$-representation $\pi$ of $G$ is called supercuspidal if it is not  isomorphic to a subquotient of a representation parabolically induced from an irreducible admissible $C$-representation of a Levi subgroup. In the established cases of the Langlands correspondence, they correspond to the irreducible continuous $C$-representations of the Galois group of $F$.   All irreducible admissible $C$-representations of $G$ are constructed from the irreducible admissible supercuspidal $C$-representations of the Levi subgroups of $G$ using parabolic induction.  

\bigskip 
 
When $c=p$, the
 finite group analogue  $\mathbf{H}(k)$  of $G$, where $\bf H$ is a connected reductive  group over a finite field $k$ of characteristic $p$, 
admits no irreducible supercuspidal $C$-representation \cite[Thm.6.12]{CabEng}. 

For $C=\mathbb F_p^{alg}$, irreducible admissible supercuspidal $C$-representations have been constructed  for  some rank $1$  groups:
 
$PGL(2,\mathbb Q_p)$  Breuil \cite{Breuil}   after the pionner work of Barthel-Livne \cite{BL};  this was the   starting point of the Langlands $p$-adic correspondance (Colmez and al.), 
 
 $PGL(2,F)$,  Paskunas  \cite{Paskunas},\cite{BreuilPaskunas},   using   $G$-equivariant coefficient systems on the adjoint Bruhat-Tits building of $G$,

 $ SL(2, \mathbb Q_p)$   \cite{Abdellatif}, \cite{Cheng}, 
 
 $ U(1,1)( \mathbb Q_p)$ unramified  \cite{Koziol}, 
 
 $   U(1,2)(F)$  unramified and $p\neq 2$  \cite{KoziolXu}.

\begin{theorem}\label{thm main} Assume $(a,c)=(0,p)$ and  $\bf G$  absolutely simple adjoint  not isomorphic to 
  the projective linear group $\bf PGL_m$ of a $d^2$-dimensional central division $F$-algebra $(m\geq 2, d\geq 1$), or to the adjoint unitary group   of a split hermitian form  in $3$ variables  over a ramified quadratic extension  of $F$. 
 
 Then, $G$ admits an irreducible admissible supercuspidal $C$-representation. \end{theorem}
  
The irreducible admissible supercuspidal $C$-representation of $G$ will be a subquotient  of a subrepresentation  $\rho^\Gamma_{C} :=C(\Gamma\backslash G)^\infty$ (smooth induction of the trivial $ C$-representation) for a discrete cocompact subgroup $\Gamma$ of $G$ \cite{BorelHarder}.
The proof is local and  uses  a criterion of supercuspidality proved
 with  R. Ollivier for  $C$ algebraically closed  \cite[Thm 3]{OV} and that we extend to any $C$:

\begin{proposition}\label{prop ss=sc}{\rm [Supercuspidality criterion]} Assume $c=p$. An irreducible admissible $C$-representation $\pi$ of $G$ is supercuspidal if and only if it contains a non-zero pro-$p$-Iwahori invariant supersingular element if and only if $\pi$ is supersingular (all pro-$p$-Iwahori invariant elements  of $\pi$ are supersingular).\end{proposition}

The equivalence supercuspidal $\Leftrightarrow $ supersingular  follows also from the  classification   \cite[Thm.9]{HenV}.   
Let  $\mathfrak B$ denote an Iwahori subgroup of $G$ and $H_C(G,\mathfrak B)$ the Iwahori Hecke $C$-algebra.

\begin{proposition}\label{prop Gammasuper}  Assume $(a,c)=(0,p)$ and $\bf G$ as in  Thm.\ref{thm main}. Then,   there exists a discrete cocompact subgroup $\Gamma$ of $G$ such that the $H_C(G,\mathfrak B)$-module $ C[\Gamma \backslash G/\mathfrak B]$ contains  a non-zero supersingular element.\end{proposition}

To prove the theorem, we  pick a non-zero supersingular  element $v$  in  $C[\Gamma \backslash G/\mathfrak B]$ and  an  irreducible quotient $\pi$ of the subrepresentation of $\rho^\Gamma_{C}  $   generated by $v$.  Then $\pi$ is  admissible ($\rho^\Gamma_{C}$ is admissible and $a=0$) and contains a non-zero $\mathfrak B$-invariant supersingular element, hence is  supercuspidal by the  supercuspidality criterion.

\bigskip  We explain now the meaning of supersingular,  why we exclude  $PGL(m,D)$ and $PU(h)$ and how we prove Prop.\ref {prop Gammasuper}. 
 
We choose  a minimal parabolic $F$-subgroup $\bf B$ of $\bf G$ containing a maximal $F$-subtorus $\bf T$.  Bruhat and Tits associated to  them  an affine Coxeter system $(W,S)$,    parameters $q_s=q^{d(s)} $ for $s\in S$, where $q$ is the residue field of $F$ and the $d(s)$ are integers $\geq 1$, and a commutative group $\Omega$  acting on the Dynkin diagram $Dyn$  of $(W,S)$  decorated with the parameters $d(s)$. The     diagram $Dyn$  is the completed Dynkin diagram of a reduced root system $\Sigma$.  
 
 The Iwahori Hecke ring $H(G^{sc}, \mathfrak B^{sc})$  of an Iwahori subgroup  $\mathfrak B^{sc}$ of  the simply connected cover $G^{sc}$ of the derived group of $G$ is  isomorphic to  the Hecke ring $H:=H(W,S,q_s)$ associated to $(W,S),(q_s)_{ s\in S}$, and the  Iwahori Hecke ring   $H(G,\mathfrak B)$   is  isomorphic to  $H:=\mathbb Z[\Omega] \ltimes  H(W,S,q_s)$. To each cocharacter $\mu\in X_*(T)$  of $T$ is associated a central element $z_\mu$ of $H$ \cite{Vigss}. Write $H_C:=C\otimes H$.  When $c=p$, an element $v$ in a right $H_C$-module is called supersingular if  $v z_\mu^n=0$ for all $\mu\in X_*(T)$ dominant but  $\mu^{-1}$ not dominant and $n\in \mathbb N$ large.

 A simple right $ H_{\mathbb C} $-module $M_\mathcal C$ which is the
  $\mathfrak B$-invariants of  an irreducible admissible square integrable modulo center $\mathbb C$-representation of $G$ is called  discrete. If  $\bf G$ is semi-simple,  $M_\mathcal C$ is discrete if and only if   the simple components of  its restriction to  $H_{\mathbb C}^{sc}$ are discrete; the equivalence uses Casselman's criterion of square integrability \cite[\S 2.5]{Cas} and the Bernstein Hecke elements \cite[Cor. 5.28]{Vig1}.
   
\medskip Assume that $\bf G$ is absolutely simple adjoint. If the type of $\Sigma$  is $A_\ell \, (\ell\geq 2)$, the parameters are equal and     $G= PGL_m(D)$ for a $d^2$-dimensional central division $F$-algebra $D$ ($m\geq 2, d\geq 1$). If the type of $\Sigma$   is $A_1$ and the two parameters  are equal, then $G=PU(h)$ is the adjoint unitary group of a split hermitian form $h$ in $3$ variables  over a ramified quadratic extension  of $F$  \cite[\S 4]{Tits}. 

When the type of $\Sigma$ is  different from  $A_\ell \, (\ell\geq 1)$ if the parameters are equal,  we find  a simple discrete   $H_{\mathbb C} $-module $M_{\mathbb C} $   with an  $ H_{\mathbb Z[q^{1/2}]}$-integral structure $M$ of  reduction,  modulo a maximal ideal $P$ of $\mathbb Z[q^{1/2}]$, a supersingular $ H_{\mathbb F_p}$-module $M_{\mathbb F_p}$:

-    When  the type   is $B_{\ell} (\ell \geq 3),C_{\ell}(\ell \geq 2) ,F_4,G_2$,  or $A_1$ with  two distinct parameters,   then $H_{\mathbb C}^{sc}$ admits a discrete character which is not the  special character \cite[5.9]{Borel}\footnote{In Borel, the Iwahori group is the fixer $\tilde Z_0\mathfrak B$ of an alcove, where $\tilde Z_0$ is the maximal compact subgroup of a minimal Levi subgroup; if $\bf G$ is $F$-split, $\tilde Z_0\mathfrak B=\mathfrak B$}.  By extending or inducing such a character, we construct  a simple discrete right $H_{\mathbb C}$-module  of dimension $\leq 2$  with a natural $H$-integral structure. 
 As we avoided the Steinberg character,  the  reduction modulo $p$ of the $H$-integral structure  is a supersingular $H_{\mathbb F_p}$-module. See  Prop. \ref{prop  12}.

-  When   the type   is $D_{\ell}  (\ell \geq 3), E_6,E_7,E_8$, the group $\bf G$ is $F$-split \cite[\S 4]{Tits}. The image  by a natural involution of $H_{\mathbb C}$   of the reflection  right $H_{\mathbb C}$-module   of $\mathbb C$-dimension $|S|$ is discrete \cite[4.23]{Lusztig}, has a  natural  $ H_{\mathbb Z[q^{1/2}]}$-integral structure of  reduction modulo a maximal ideal $P$ of $\mathbb Z[q^{1/2}]$ a    supersingular  $ H_{\mathbb F_p}$-module.  See Prop. \ref{prop discsinc}.

Now, we find  $\Gamma$ such that   $ C[\Gamma \backslash G/\mathfrak B]$ contains  a non-zero supersingular element.
The existence of discrete cocompact subgroups $\Gamma$ in $G$  is  ensured by $a=0$   \cite{Margulis}.  By the $p$-adic version of the de George-Wallach limit multiplicity formula (\cite[Appendice3, Prop.]{DKV} plus \cite[Thm.K]{Kazhdan}),  the simple discrete $H_{\mathbb C}$-module  $M_\mathbb C$ embeds in   $\mathbb C[\Gamma \backslash G/\mathfrak B]$ for some discrete cocompact subgroup  $\Gamma$ of $G$.  All the  $ H_{\mathbb Z[q^{1/2}]}$-integral structures of  $M_\mathbb C$ have the same semi-simplification hence have supersingular reduction modulo $P $, and   $M':= M_\mathbb C \cap  \mathbb Z[q^{1/2}][\Gamma \backslash G/\mathfrak B]$ is another $ H_{\mathbb Z[q^{1/2}]}$-integral structure of $M_\mathbb C$ of reduction modulo $P$ a  $ H_{\mathbb F_p}$-submodule $M'_{\mathbb F_p}$ of $\mathbb F_p[\Gamma \backslash G/\mathfrak B]$.  See Prop. \ref{prop mod}.
The scalar extension from  $\mathbb F_p$ to $C$ of the supersingular $ H_{\mathbb F_p}$-module $M'_{\mathbb F_p}$ is  a non-zero supersingular $H_C$-submodule of $C[\Gamma \backslash G/\mathfrak B]$. This ends the proof of Prop.\ref {prop Gammasuper}.

A similar argument  for an irreducible admissible supercuspidal complex representation of $G$ produces an integral structure  with admissible reduction (Cor. \ref{cor sca}).

\bigskip Returning to $\bf G$ general,  Kret proved that  $G$ admits  irreducible  admissible supercuspidal   complex representations  \cite{Kret} (if $a=0$ \cite{BP}). We extend this to  
 any field  $C$ of characteristic $0$   in \S\ref{S 8}:

  \begin{proposition}\label{prop Cfinite} {\rm [Change of coefficient field]} 
If $G$ admits an irreducible admissible  supercuspidal  representation over  some field  of characteristic $c$, supposed to be  finite if $c=p$, 
then $G$ admits an irreducible admissible  supercuspidal  representation over any field   of characteristic $c$.  
    \end{proposition}
We reduce  the construction of an irreducible admissible supercuspidal representation of $G$  to the case of $\bf G$ absolutely simple and adjoint in \S\ref{S 2}:
\begin{proposition} \label{prop adj} {\rm [Reduction to a simple adjoint group]} 
   Assume  $a=0$ and $C$ algebraically closed  or   finite.  If     any connected absolutely simple adjoint $p$-adic  group  admits an irreducible admissible supercuspidal $C$-representation, then any connected reductive  $p$-adic group 
  admits an irreducible admissible supercuspidal $C$-representation. \end{proposition}
  Summarizing our results,  we obtain:
   \begin{theorem} Assume  $(a,c)=(0,p)$.  Then  $G$ admits  an irreducible admissible supercuspidal  $C$-representation, if:
   
    $PGL(m,D)$ and   $PU(h)$ (as in Thm.\ref{thm main}) admit  an irreducible admissible supercuspidal   representation over some finite field of characteristic $p$.
  \end{theorem}

When $(a,c)=(0,p)$,  $PGL(m,D)$ and   $PU(h)$ should also admit  an irreducible admissible supercuspidal  $C$-representation.
 We miss  $PGL(m,D)$ and   $PU(h)$ because we do not know 
the  integrality properties of the unramified irreducible admissible complex representations  of $G$ corresponding to the integral reflections modules of the generalized affine Hecke algebras of supersingular reduction   modulo $p$. 
 The  missing cases will probably be completed by Herzig with a global method  for $PGL(m,D)$ and by Koziol with coefficient systems on the tree for   $PU(h)$.  
  
 \medskip 
{\it Ackowledgements}   I thank Pioline, Savin and the organizers of the conference on automorphic forms and string theory in Banff (2017-10-30) for emails, discussions and their invitation   who led me to look closely at   unramified minimal representations corresponding to the reflection modules  of the affine Hecke algebras \cite[Thm. 12.3]{GS}. Minimal representations  are important for string theorists   \cite{KPW}   and for mathematicians as they allow to construct analogues  of theta series. Their relation to supersingular representations, a complete surprise for me,  is at the origin of this work on a basic question raised  fifteen years ago  for modulo $p$ representations. 

I thank also Heiermann for a discussion on discrete modules,  Harder for  emails on discrete cocompact subgroups, Waldspurger for reminding me the antipode and a reference,   Henniart for providing the proof of Prop.\ref{prop GtimesH} b),   Herzig for writing  a  draft of his global method for  $GL(m,F)$, Herzig and Koziol for  working on   the missing cases.

%\cite[II.5.2 Prop.14]{SerreCohGal}, \cite[Thm.6.14]{Platonov-Rapinchuk}. 

 \section{Iwahori Hecke ring } 
 
We review in this section the parts of the theory of Iwahori Hecke algebras \cite{Vig1}, \cite{Vigcenter},  \cite{Vigss}  appearing in  the proof of the key proposition \ref{prop Gammasuper}.

Let $\bf G$ be a reductive   connected $F$-group,  $\bf T$   a  maximal $F$-split subtorus of $\bf G$,  $\bf B $ a minimal $F$-parabolic subgroup of $\bf G$ containing $\bf T$, and $x_0$ a   special point of the apartment of  the adjoint Bruhat-Tits buiding defined by $\bf T$. 
Associated to the triple  ($\bf G$, $\bf T$, $\bf B $, $x_0$) are:
the center  $\bf Z(G)$  of $\bf G$,  the root system $\Phi$, the set of simple roots $\Delta$, the  $\bf G$-centralizer  $\bf Z$   of  $\bf T$,  the normalizer $\bf N$, the  unipotent radical $\bf U$  of   $\bf B $, (hence ${\bf B}=\bf Z \bf U$),    the triples  $({\bf G^{sc}}, {\bf T^{sc}}, {\bf B^{sc}}) $ and   $({\bf G^{ad}}, {\bf T^{ad}}, {\bf B^{ad}}) $ for the simply connected covering and the adjoint group of the derived subgroup of $\bf G$,   the  natural homomorphisms $G^{sc}:={\bf G^{sc}}(F)\xrightarrow{ i_{sc}}G={\bf G}(F)\xrightarrow{ i^{ad}}G^{ad}={\bf G^{ad}}(F)$, an alcove $\mathcal C$ of the apartment with vertex $x_0$,  the Iwahori subgroups  $\mathfrak B, \mathfrak B^{sc},\mathfrak B^{ad}$   of $G,  G^{sc}, G^{ad}$  fixing  $\mathcal C$, the  Iwahori Hecke ring  $H (G, \mathfrak B):=\End_{\mathbb Z G}\mathbb Z[\mathfrak B \backslash G] $ and similarly for the simply connected and adjoint semi-simple groups.

 The  natural ring homomomorphism $H(G^{sc}, \mathfrak B^{sc}) \to H(G , \mathfrak B )
$ induced by  $i_{sc}$  is injective and we identify $H(G^{sc}, \mathfrak B^{sc}) $ with a subring of $H(G , \mathfrak B )$.
   There is a  canonical   isomorphism  $H(G^{sc}, \mathfrak B^{sc})  \xrightarrow{j^{sc}} H (W,S, q_s)$ where 
$H:=H (W,S, q_s) $ is the Hecke ring of  an  affine Coxeter system $(W,S)$ with  parameters  $(q_s=q^{d(s)})_{s\in S} $ where $q$ is the residue field of $F$ and the $d(s)$ are integers $\geq 1$.
The  Dynkin diagram of   $(W,S)$  is the completed Dynkin diagram $Dyn$ of a  reduced root system $\Sigma$. The  image    $\Omega=\Omega_G $  of the Kottwitz morphism of $G$ acts  on  the Dynkin diagram $Dyn$ decorated  with the parameters $(d_s)_{s\in S} $, and the isomorphism $j^{sc}$ extends to an isomorphism
 \begin{equation}\label{eq Had}H (G, \mathfrak B) \xrightarrow{j} \mathbb Z[\Omega] \ltimes  H (W,S, q_s),  
 \end{equation} 
  The Iwahori Hecke ring of  $G $ is  determined by the type of $\Sigma$, the parameters  and the group $\Omega$ acting on the decorated  Dynkin diagram $Dyn$. The quotient of   $\Omega $ by the fixer of  $\mathcal C$ in $\Omega$ is isomorphic to a subgroup $\Psi$ of the group of automorphisms $\Aut(W,S,d_s) $ of the decorated Dynkin diagram.   
  The  generalized affine Hecke ring $\tilde H:=\mathbb Z[\Omega] \ltimes  H (W,S, q_s) 
$   is a free $\mathbb Z$-module of basis $T_w$ for $w\in  \tilde W :=W\rtimes \Omega$  satisfying the braid  and quadratic relations:
$$
  T_w T_{w'}=T_{ww' } \  \ \text{for} \ \  w,w'\in  \tilde W, \  \ell(w)+\ell(w')=\ell(ww'),
$$ 
\begin{equation}\label{eq Tsq}(T_{s}- q_sT_1)(T_{s}+T_1)=0  \ \  \text{for } \  s\in S, 
  \end{equation}
where $\ell$ is the length on $W$ associated to $S$ extended to $\tilde W $  by $\ell(uw)=\ell(wu)=\ell(w)$ for $u\in \Omega, w\in W$.  
The linear map $T_w\mapsto (-1)^{\ell (w)}T_w^*$ for $w\in \tilde W$,  where for $w=us_1\ldots s_n $ with $ u\in \Omega, s_i\in S,n=\ell(w)$,  
\begin{equation}T_w T^*_{w^{-1}}=q_w, \ \text{where} \ q_w:=q_{s_1}\ldots q_{s_n}, \quad T_s^*:=T_s-q_s+1, \quad T_w^*:= T_u T_{s_1}^*\ldots T_{s_n}^* ,
 \end{equation}
   is an automorphism of  $\tilde H$.

  The   unique parahoric subgroup $Z_0$ of  $Z:=\bf Z(F)$ is contained in the  maximal compact subgroup $\tilde Z_0$.  When $\bf G$ is $F$-split or semi-simple simply connected, $Z_0=\tilde Z_0$. The group  $N/Z_0$  is isomorphic to $\tilde W$,   acts on the apartment and forms a system of representatives of the double classes of $G$ modulo $\mathfrak B$.
The subgroup $\Lambda=Z/Z_0$ of $N/Z_0$ is  commutative finitely generated  of torsion $\tilde Z_0/Z_0$ acting by translation on the apartment,  the quotient map $N/Z_0\to N/Z$ splits, identifying   the (finite) Weyl group $W_0$ of $\Sigma$ with the fixer in $W$ of a special vertex of the alcove  $\mathcal C$.  The semi-direct product $\Lambda \rtimes W_0$ is  equal to $\tilde W$ and  
$\Lambda^{sc}\rtimes W_0=W$ where  $ \Lambda^{sc}:= \Lambda \cap W$. 
An element  $\lambda\in \Lambda $ is called dominant  (and $\lambda^{-1}$ anti-dominant),  if $z(U\cap  {\mathfrak B}) z^{-1} \subset (U\cap  {\mathfrak B}) $ for $z\in Z$ lifting $\lambda$. The dominant monoid $\Lambda^+$ consists of dominant elements of  $\Lambda$.
 
The  cocharacter group $ X_*(T)$ of $T$ is isomorphic to  $T/T_0 \simeq TZ_0/Z_0=\Lambda_T$    by the  $W_0$-equivariant map $\mu \mapsto \lambda_\mu:=\mu(p_F)Z_0/Z_0 $ for a  fixed  an uniformizer $p_F$ of $F$. A cocharacter
$\mu $ is  dominant   if $\lambda_\mu$ is.
For  $\lambda\in 
\Lambda$ and   $n\in \mathbb N$ large, $\lambda^n \in \Lambda_T$, and $ \langle \chi, \lambda ^n\rangle$ is defined for any character $\chi \in X^*(T)$.   The fixer of the alcove  $\mathcal C$ in $\Lambda$  is  $\Lambda_{Z(G)\tilde Z_0}=Z(G)\tilde Z_0/ Z_0$ and

 $\Lambda_{Z(G)\tilde Z_0} =\{\lambda\in \Lambda \ | \  \langle \alpha, \lambda ^n\rangle=0 \ \text{ $n$ large such that $ \lambda^n \in \Lambda_T$ and } \ \alpha \in \Delta\}$. 
 
 \begin{lemma}\label{lem fix}
  The fixer of the alcove $\mathcal C$ in $\Omega$  is  $\Lambda_{Z(G)\tilde Z_0}$.
   \end{lemma}
 
 \begin{proof}
Write $u\in \Omega$ as  $u=\lambda w_0$ with $\lambda\in \Lambda, w_0\in W_0$. As $w_0$ fix a special point of $\mathcal C$, $\lambda$ does also. As $\lambda$ acts by translation on the apartment, $\lambda$ fixes $\mathcal C$. As  $\lambda$ and $u$ fix $\mathcal C$,  $w_0$ fixes also $\mathcal C$ hence $w_0=1$. \end{proof}

The  action of $\Omega$ on the alcove $\mathcal C$ and  the embeddings of $\Lambda$ and $\Omega$ into  $\tilde W$ induce isomorphisms 
\begin{equation}\label{eq fix} \Omega /\Lambda_{Z(G)\tilde Z_0 } \to \Psi, \quad  \Lambda /(\Lambda_{Z(G)\tilde Z_0} \times  \Lambda^{sc})
 \to \tilde W /(\Lambda_{Z(G)\tilde Z_0} \times  W) \leftarrow \Omega/\Lambda_{Z(G)\tilde Z_0 }.
 \end{equation}

 \begin{lemma}\label{lem fin} The subgroup $\Lambda_{Z(G) } \times  \Lambda^{sc}$ of $\Lambda$ is finitely generated of finite index. 
 
 The submonoid $\Lambda_{Z(G) } \times  \Lambda_{sc}^+$ of  the dominant monoid $\Lambda^+$  is finitely generated of finite index. 
 \end{lemma}
 \begin{proof} The commutative group $\Lambda_{Z(G) } \times  \Lambda^{sc}$ is finitely generated and a finite index subgroup of $\Lambda$, as  $\tilde Z_0/Z_0$ is finite and 
  \eqref{eq fix} implies that  $\Lambda /( \Lambda_{Z(G)\tilde Z_0 } \times  \Lambda^{sc})\simeq \Psi$. Gordan's lemma implies  the second assertion (as in the proof of \cite[7.2 Lemma]{HV1}). 
\end{proof}

 The   $\tilde W$-conjugacy class of  $\lambda \in \Lambda$  is the $W_0$-orbit   of $\lambda$.   A basis of the center of $ \tilde  H $ \cite[Thm.1.2]{Vigcenter} is
$$\sum_{\lambda\in \mathcal O} E_{\lambda} \quad \text{for the  $W_0$-orbits } \ \mathcal O \subset \Lambda,$$
where  $E_\lambda $  for  $\lambda\in \Lambda$ are the integral Bernstein elements    of $\tilde H$ \cite[Cor. 5.28, Ex.5.30]{Vig1}:
  \begin{align}\label{eq	 E}E_{\lambda}=\begin{cases} T_\lambda \   \text{if}  \   \lambda \ \text{ is anti-dominant}\\  T_\lambda^* \   \text{if}  \   \lambda  \ \text{is dominant}
  \end{cases},  \quad E_{\lambda_1} E_{\lambda_2} = (q_{\lambda_1}q_{\lambda_2}q_{\lambda_1 \lambda_2}^{-1})^{1/2} E_{\lambda_1 \lambda_2}  \ \text{for} \  \lambda_1, \lambda_2\in \Lambda.
\end{align}
When $\lambda_1, \lambda_2$ are both dominant (or anti-dominant), $ E_{\lambda_1} E_{\lambda_2} =E_{\lambda_1 \lambda_2} $.
 For $\mu\in X_*(T)$, let  $\mathcal O_\mu$  denote    the $W_0$-orbit  of $\lambda_\mu$  and  write
 $z_\mu:= \sum_{\lambda\in \mathcal O_\mu} E_{\lambda} $.  Any $W_0$-orbit $\mathcal O_\mu$ contains a unique dominant (resp. anti-dominant) cocharacter $\mu^+$  (resp. $\mu^-$), and  $z_\mu=z_{\mu^+}=z_{\mu^-}$.    
 
 The   invertible elements in   the  dominant monoids  $\Lambda^+$ and $ X_*(T)^+$ are the subgroups $\Lambda_{Z(G)\tilde Z_0}$ and $ \cap _{\alpha\in \Delta}\Ker \alpha$.    Only the trivial element is invertible in
   the dominant monoid $\Lambda_{sc}^+$. Write  $\Lambda_{Z(G)}:=Z(G)Z_0/Z_0$.

   The generalized affine ring $\tilde H$ contains the commutative subring $\mathcal A$ of  $\mathbb Z$-basis
   $(E_{\lambda})_{\lambda\in \Lambda}$. When $G=Z$,  the Bernstein elements are simply the classical elements $T^Z_\lambda$, and the Iwahori Hecke ring  $H (Z,Z_0)$ is isomorphic to $\mathbb Z[\Lambda]$.  In general, $\mathcal A$ is not isomorphic to $\mathbb Z[\Lambda]$ but   the subring  $\mathcal A^+$ of basis $(E_{\lambda})_{\lambda\in \Lambda^+}$ is   isomorphic to $\mathbb Z[\Lambda^+]$.
Denote by  $\mathcal A^{sc},  \mathcal A_T,  \mathcal A_{sc}^+, \mathcal A^+_T$ the subrings of  respective bases $(E_{\lambda})_{\lambda \in  \Lambda_{sc}}$, $(E_{\lambda_\mu})_{\mu\in X_*(T)}$, $(E_{\lambda})_{\lambda\in \Lambda_{sc}^+}$,   $(E_{\lambda_\mu})_{\mu\in X_*(T)^+}$.

 \begin{lemma}\label{lem fixA}   $\mathcal A $  is  finitely generated as an  $\mathcal A_T$-module and  as a  $ \mathbb Z[\Lambda_{Z(G)} ]\times \mathcal A_{sc} $-module. 
  
  $\mathcal A^+ $  is   finitely generated  as an $\mathcal A_T^+$-module and as a $ \mathbb Z[\Lambda_{Z(G) }]\times \mathcal A_{sc}^+ $-module.  
    \end{lemma}
\begin{proof} For $T$  \cite[Lemma 2.14, 2.15]{Vigcenter}. Otherwise, Lemma \ref{lem fin}. \end{proof}

\section{Supersingular and discrete modules}
Let $C$ be a field of characteristic $c$ and the  Iwahori Hecke $C$-algebras 
$ H_C(G^{sc},\mathfrak B^{sc}):=C\otimes_{\mathbb Z} H(G^{sc},\mathfrak B^{sc})$  and $ H_C(G,\mathfrak B):=C\otimes_{\mathbb Z} H(G,\mathfrak B)$, 
  isomorphic to the   affine  and generalized affine Hecke $C$-algebras   $H_C=C\otimes_{\mathbb Z} H  (W,S, q_s)$ and $\tilde H_C= C[\Omega] \ltimes_{\mathbb Z}  H  (W,S, q_s)$. 
  
  This section  introduces the supersingular  right $\tilde H_C$-modules when $c=p$ (there are none  when $c\neq p$) and   the discrete simple  right $ H_{\mathbb C}(G,\mathfrak B)$ modules.

 \begin{definition} Let $M$ be  a non-zero right $\tilde  H_C $-module. 
An element $v\in M$ is called supersingular if and only if $v z_\mu^n=0$  for all dominant  $\mu\in X_*(T) $ with $\mu^{-1}$ not dominant, and some large positive integer $n$.
The $ \tilde H_C $-module $M$ is called supersingular when all its elements are supersingular\footnote{In \cite[Def. 6.10]{Vigss} there is a different definition:  there exists $n>0$ with $M z_\mu^n=0$ for all $\mu$ not invertible in $X_*(T)^+$}.
\end{definition}
 
\begin{remark} {\rm   $v$ is supersingular  if and only if  $v z_\lambda^n=0$ for any $\lambda \in  \Lambda \setminus \Lambda_{Z(G)\tilde Z_0}$ and  large $n$. 
We can restrict to $\mu$ (or $\lambda$) dominant, or anti-dominant.  }
\end{remark}
 
\begin{fact}\label{fact}   

-  A simple $ \tilde H _C $-module $M$  is finite dimensional, and is semi-simple as an  $ H_C $-module.

-  A  $ \tilde  H _C $-module is supersingular if and only if its restriction to  $H_C   $ is supersingular   \cite[Cor.6.13]{Vigss}.

-   The simple  supersingular $H_C $-modules are\footnote{There are no non-zero supersingular modules if $c\neq p$} the characters  which are not special or trivial   {\rm (see the next section)} when $C$ is     ``large''  of characteristic $c=p$  \cite[Cor.6.13]{Vigss}. 
\end{fact} 

  \bigskip We denote by $\Mod_C(G,\mathfrak B)$  and $\Mod_C(H(G,\mathfrak B))$ the  categories of $ C$-representations of $G$ generated by their $\mathfrak B$-invariant vectors and  of right  $ H_{ C}(G,\mathfrak B)$-modules. 
   The  $\mathfrak B$-invariant functor $\pi \mapsto \pi^{\mathfrak B}:\Mod_C(G,\mathfrak B)\to \Mod_C(H(G,\mathfrak B))$  has  a left adjoint  $\mathfrak T: \tau \mapsto 
\tau \otimes_{ H_{ C}(G,\mathfrak B)}  C[\mathfrak B\backslash G]$. 

\begin{fact}\label{fact1}

 When $c\neq p$, the functor $\pi \mapsto \pi^{\mathfrak B}$ induces a bijection between the isomorphism classes of the irreducible $C$-representations $\pi$ of $G$ with  $\pi^{\mathfrak B} \neq 0$ and of the simple right  $H_C(G,\mathfrak B)$-modules  \cite[I.6.3]{Viglivre}.  
When $C=\mathbb C$, the functors are inverse equivalences of categories {\rm (Bernstein-Borel-Casselman)}.
\end{fact}

\medskip Let  $\pi$ be an  irreducible   complex representation  of $G$ with   $\pi^{\mathfrak B} \neq 0$. We recall the   classical properties of    $\pi$, including the Cassleman's criterion of square integrability modulo center, before giving the definition of a discrete  simple  right $ H_\mathbb C(G,\mathfrak B)$-module.

\begin{fact}\label{fact2}
- a)    $\pi$ is isomorphic to a subrepresentation of $\ind_B^G\sigma$ where $\sigma$ is a   $\mathbb C$-character  of $Z$ trivial on $\tilde Z_0$. 

- b)   The representation of $Z$ on the $U$-coinvariants
 $(\ind_B^G\sigma)_U$ is semi-simple, trivial on $\tilde Z_0$ and contains a subrepresentation isomorphic to  $\pi_U$.

- c)  The quotient map $f: \pi\mapsto \delta_B^{-1/2}\pi_U$ induces an $H_{\mathbb Z[q^{-1/2}]} (Z,Z_0)$-equivariant  isomorphism  $\pi^\mathfrak B\to \pi_U^{Z_0}(=\pi_U)$ for the Bernstein  $\mathbb Z[q^{-1/2}]$-algebra embedding  
$$H_{\mathbb Z[q^{-1/2}]} (Z,Z_0) \xrightarrow{t_B} H _{\mathbb Z[q^{-1/2}]} (G, \mathfrak B) \quad t_B(T^Z_\lambda)=\theta_\lambda:=q_\lambda^{-1/2}E(\lambda) \quad {\text for\ } \lambda \in \Lambda, $$
that is, $f (v \theta_\lambda)= f(v) T^Z_{\lambda }$ for $\lambda \in \Lambda , v\in \pi^{\mathfrak B}$  \cite[II.10.1]{VigSelecta}.
{\rm Note that $q_{\lambda}=\delta_B(z)$ where  $\delta_B$ is the modulus  of $B$ and $ z\in Z$ of image $\lambda\in \Lambda$ and  $\theta_{\lambda_1} \theta_{\lambda_2}=\theta_{\lambda_1\lambda_2}$ for $\lambda_1,\lambda_2\in \Lambda$. }

    -  d) Casselman's criterion: 
  $\pi $  is square integrable modulo center    \cite[\S 2.5]{Cas} if and only if its central character is unitary and \begin{align*} |\chi  (\mu(p_F))|& \leq 1  \ \text{for all anti-dominant  }\  \mu \in X_*(T)   \text{ but }  \mu ^{-1}   \text{ not anti-dominant, }
  \end{align*} 
 for any character  $\chi$ of  $Z$ contained in $\delta_B^{-1/2}\pi _U$  \cite[Thm. 6.5.1]{Cas}.
 \end{fact} 
  \begin{definition} A  simple  right $H_ {\mathbb C}(G,\mathfrak B)$-module    is called discrete when it is   isomorphic to $\pi^{\mathfrak B} $ for an irreducible admissible square integrable modulo center  $\mathbb C$-representation $\pi$ of $G$.   \end{definition}
 \begin{proposition} A  simple  right $H_ {\mathbb C}(G,\mathfrak B)$-module  $M$  is discrete if and only if   any complex character $\chi$ of $\mathcal A $ contained in  $M$ satisfies: the  restriction of $\chi$  to  $\Lambda_{Z(G)} $ is  unitary  and   
$$ |\chi  (\theta_\mu)| \leq 1  \ \text{for  any dominant  }\  \mu \in X_*(T) \text{ but }  \mu ^{-1}   \text{ not dominant, } \quad  \theta_\mu:=\theta_{\lambda_\mu}
$$
.\end{proposition}
\begin{proof}  
 Let $\chi:Z\to \mathbb C^*$  be a character trivial on $Z_0$. Writing $\chi(z)=\chi(\lambda)$ for $z \in Z$ of image $\lambda \in \Lambda$ and noting $f(v) T^Z_{\lambda }=z^{-1} f(v)$,
we have in c) $$z^{-1} f(v)=\chi(z^{-1}) f(v)\Leftrightarrow f (v \theta(\lambda ))=\chi(\lambda^{-1}) f(v) \Leftrightarrow v \theta(\lambda )=\chi(\lambda^{-1})v .$$
Hence $\chi$ is contained  in $\delta_B^{-1/2}\pi_U $ if and only  $\chi^{-1} :
\mathcal A \to  \mathbb C, \ \chi^{-1}( \theta_\lambda):= \chi^{-1}(\lambda),$ is contained $\pi^\mathfrak B$.
Apply Casselman's criterion (Fact \ref{fact2} d)) (the inverse of an anti-dominant element is dominant).
   \end{proof} 
\begin{remark}\label{rem danger} {\rm  Some authors see $\pi^\mathfrak B$ as a left $H_ {\mathbb C}(G,\mathfrak B)$-module. One exchanges ``left'' and ``right''  by putting $T_w v = vT_{w^{-1}}$ for $w\in \tilde W, v\in \pi^\mathfrak B$.  The left or right  $H_ {\mathbb C}(G,\mathfrak B)$-module  $\pi^\mathfrak B$ is called discrete if $\pi$ is square integrable modulo center. For left modules, the proposition  holds true with anti-dominant instead of   dominant.}
\end{remark}
\begin{lemma} \label{lem chi}For  a  character $\chi:\mathcal A\to \mathbb C$, the  following properties are equivalent:

(i) $\chi|_{\Lambda_{Z(G)}} $ is  unitary  and   
$ |\chi  (\theta_\mu)| \leq 1 $ for  any dominant  $\mu \in X_*(T)$ with $\mu^{-1}$ not dominant. 

(ii) $ |\chi  (\theta_\lambda)| \leq 1$ for any  dominant $ \lambda \in \Lambda$.  

(iii) $ |\chi  (\theta_\lambda)| =1$ for any $\lambda \in \Lambda_{Z(G)}$ and $ |\chi  (\theta_\lambda)| \leq 1$ for any  dominant $ \lambda \in \Lambda_{sc}$.  
\end{lemma}
\begin{proof}   
 (ii) implies (i) because $ |\chi  (\theta_\lambda)| \leq 1$ for any   $ \lambda \in \Lambda^+$,  implies   $ |\chi  (\theta_\lambda)| = 1$ when $\lambda$ is invertible in $\Lambda^+$, i.e. 
$\lambda \in \Lambda_{Z(G)\tilde Z_0}$.  Conversely, (i) implies that for $\lambda \in \Lambda^+$ and $n\in \mathbb N$ large with $\lambda ^n\in \Lambda_T^+$   we have 
  $|\chi(\theta_\lambda^n)|\leq 1$. This implies $|\chi(\theta_\lambda)|\leq 1$ hence  (ii).
  The arguments of the equivalence between  (i) and  (iii) are similar using Lemma \ref{lem fin}.
  \end{proof}

\begin{proposition} \label{prop dis}  A  simple  right $H_ {\mathbb C}(G,\mathfrak B)$-module $M$  is discrete if and only if $\Lambda_{Z(G) }$ acts on $M$ by a  unitary character and 
the simple components of $M$ restricted  to $H_ {\mathbb C}(G^{sc},\mathfrak B^{sc})$ are discrete.
 \end{proposition}
\begin{proof}  Lemma \ref{lem chi} (iii).
\end{proof}

\section{Characters}\label{S 4} In this section,  $C$ is a field of characteristic $c$ and $\bf G$ is absolutely  simple. We give  the characters  $H \to C$  which extend to $\tilde H $. This is an exercice, already  in the litterature  when $C=\mathbb C$ is the complex field  \cite{Borel}.  

 For distinct $s,t\in S$, the order $n_{s,t}$ of $st$ is finite except if the type is $A_1$. In the finite case,
\begin{align}\label{eq TsTt}(T_sT_t)^r=(T_tT_s)^r \  \text{if } n_{s,t}=2r, \quad 
(T_tT_s)^rT_t=T_t(T_sT_t)^r \    \text{if  } n_{s,t}=2r+1. 
\end{align}
The  $T_s$ for $s\in S$ and the relations  \eqref{eq Tsq}, \eqref{eq TsTt}  give a  presentation of $H $. A presentation of  $\tilde H $ is given by
the  $T_u, T_s$ for $u\in \Omega, s\in S$  and the  relations \eqref{eq Tsq}, \eqref{eq TsTt} and
\begin{align} \label{eq TsTu}
 T_uT_{u'}=T_{uu'}, T_uT_s=T_{u(s)}T_u   \ \text{for}  \ u\in \Omega, s\in S.
 \end{align}
 We have a disjoint decomposition $S=\sqcup _{i=1}^m S_i$ where $S_i$ is the intersection of $S$ with a conjugation class  of $W$ and \cite[3.3]{Borel}: $$\begin{matrix} m&=& 1 & \text{\ type of \ } \Sigma= & A_\ell (\ell \geq 2), D_{\ell} (\ell \geq 2), E_6, E_7, E_8\\
         && 2  &  &A_1, B_{\ell} (\ell \geq 3), F_4, G_2\\
 &&  3 & & C_\ell  (\ell \geq 2)
          \end{matrix}
$$
When $m>1$, we fix a numerotation such that $|S_1|>  |S_2|$ when $m=2$ (except for $A_1$ where $S_1=\{s_1\}, S_1=\{s_2\}$)  and $|S_1|>  |S_2|=|S_3|$ when $m=3$. We have 
$S_3=\{s_3\}$ for $C_\ell$, $S_2=\{s_2\}$  for $A_1, B_\ell, C_\ell, G_2 $ and $S_2$ has  two elements for $F_4$, and $S_1 $  has respectively $\ell, \ell -1, 3, 2$ elements for $B_\ell, C_\ell, F_4, G_2$. The parameters $d_s$ are equal on $S_i$ to a positive  integer   $d_i$.  The automorphism group  $\Aut(W,S)$  of the completed Dynkin diagram $Dyn$ is \cite[1.8]{IwaMats}: 
 $$\begin{matrix}       \Aut(W,S) &=&\mathbb Z/(\ell+1) \mathbb Z  &\text{\ type of \ }  \Sigma=&  A_\ell (\ell \geq 1)\\
  & &\mathbb Z/ 2\mathbb Z&     &  B_{\ell} (\ell \geq 3), C_\ell  (\ell \geq 2), E_7 \\
                  &&(\mathbb Z/ 2  \mathbb Z) \times (\mathbb Z/ 2  \mathbb Z) &  & D_{2\ell} (2\ell \geq 2)\\
                  & &\mathbb Z/ 4\mathbb Z&  &  D_{2\ell+1} (2\ell+1 \geq 3)\\
 & &\mathbb Z/ 3\mathbb Z&   &    E_6\\ 
   & &\{1\} &   & E_8, F_4, G_2
          \end{matrix}
          $$

\begin{lemma}\label{lem aut} The  automorphism  group $\Aut(W,S,d_i) $  of the decorated Dynkin diagram   $Dyn$ is equal to $\Aut(W,S)$ with two  exceptions:   type  $A_1$ and different parameters $d_1\neq d_2$,  type  $C_\ell (\ell \geq 2) $ and different parameters $d_2\neq d_3$, where  $\Aut(W,S,d_i)=\{1\} , \Aut(W,S)=\mathbb Z/ 2\mathbb Z$.  
 \end{lemma}
\begin{proof}
 When $m=2$, $S_1$ and $S_2$ are $  \Aut(W,S)$-stable except for 
  $A_1$ where the non-trivial element of $  \Aut(W,S)$ permutes $s_1,s_2$;  $\Aut(W,S,d_i) $  fixes $s_1$ and $s_2$ if $d_1\neq d_2$.
 
   When $m=3$, the type is $C_\ell (\ell \geq 2) $,  the non-trivial element of $  \Aut(W,S)$ permutes $s_2,s_3$;   $  \Aut(W,S, d_i)$  fixes $s_2$ and $s_3$ if $d_2\neq d_3$.
   \end{proof}

\begin{lemma} \label{lem Bor}
  A map $(T_s)_{s\in S}\to C$  is the restriction of a character  $\chi: H \to C$ if and only if:

when  $c\neq p$, it is constant and equal to $\chi_i=-1$ or $q^{d_i}$    on each $S_i$. There are $2^m$  characters if $q+1\neq 0$ in $C$.

when  $c=p$,  its   values  are $-1$ or $0$. There are $2^{|S|}$ characters.
\end{lemma}
\begin{proof} This follows from the presentation  of $ H $. When  $c\neq p$, the  $T_w$ are  invertible.      When $c=p$ \cite[Prop.2.2]{Vigss}.
   \end{proof}
  The unique character $\chi:H \to C$ with $ \chi (T_s)=q_s$ (resp. $ \chi (T_s)=-1$) for all $ s\in S$  is   called the trivial  (resp. special) $C$-character. If they are equal then $c\neq p, 0$.

        A  character $\chi:H \to C $ extends to a  character  of $\tilde H $ if and only if $\chi (T_s)=\chi (T_{u(s)})$ for all $s\in S$ and $u\in  \Omega$.  When the image  $\Psi$ of $\Omega$ in $\Aut(W,S,d_i)$ is trivial or when $m=1$, any character of $H$ extends to $\tilde H$. The extensions are not unique in general.
The  trivial and special  characters extend, their extensions  are also called   trivial and special.

\begin{lemma} \label{lem ext}Assume $c\neq p$.
A   character  $\chi= (\chi_i)_{1\leq i \leq m}: H \to C $   extends to a character of  $\tilde H $ except for

- type $A_1$, equal parameters $d_1=d_2$, $\Psi   \simeq \mathbb Z/ 2 \mathbb Z$,   when $\chi_1\neq \chi_2$,

- type  
$C_\ell (\ell \geq 2) $, equal parameters $d_2=d_3$,   $\Psi   \simeq \mathbb Z/ 2 \mathbb Z$,  when $\chi_2\neq \chi_3$.

  \end{lemma}

\begin{proof}   Assume $m>1$ and  $\Psi$ not trivial. Then  $\Psi  \simeq \Aut(W,S) \simeq \mathbb Z/ 2 \mathbb Z$ with three cases:
  
 Type  $A_1$ and equal parameters   $d_1= d_2$, say $d$. Then  $\Psi$ permutes  $s_1, s_2$. Only the special and trivial characters of  $H$ extends to $\tilde H$.

 Type $C_\ell (\ell \geq 2)$ and equal  parameters $d_2=d_3$, say $d$. Then $\Psi$ permutes $s_2,s_3$,  only the $C$-characters  of $H$ with $\chi_2=\chi_3$   extend to $\tilde H$.

Type $B_\ell (\ell \geq 3)$. Then, $\Psi$ stabilizes  $S_1$ and $S_2$ hence  all  $C$-characters of $H$ extend to $\tilde H$.
\end{proof} 

We combine these results in a proposition:
\begin{proposition} \label{prop cha} (i)  $H_\mathbb C $ admits $2^m$  characters $\chi$, they are all  $\mathbb Z$-integral, and their reduction modulo $p$  are supersingular  except for the special and trivial character.
  
(ii)   $\tilde H_\mathbb C$ admits a  $\mathbb Z$-integral  simple (left or right) module of supersingular reduction modulo $p$, and restriction $\chi \oplus \overline \chi$ to $H_\mathbb C$, when 

$\chi=(\chi_1, \chi_2), \overline \chi=(\chi_2, \chi_1)$,  type $A_1$,  $d_1=d_2$,    $\Psi   \simeq \mathbb Z/ 2 \mathbb Z$, and $\chi_1\neq \chi_2$,

 $\chi=(\chi_1, \chi_2, \chi_3), \overline \chi =(\chi_1, \chi_3, \chi_2)$,  type $C_\ell (\ell \geq 2) $, $d_2=d_3$,   $\Psi   \simeq \mathbb Z/ 2 \mathbb Z$, and $\chi_2\neq \chi_3$.

 (iii)  Otherwise the characters $\chi$ of  $H_\mathbb C$ extends to  $\mathbb Z$-integral complex characters of  $\tilde H_\mathbb C$, of supersingular reduction modulo $p$ if $\chi$ is not special or trivial.
   \end{proposition}
\begin{proof} (i) Lemma \ref{lem Bor}, Fact \ref{fact}. The reduction modulo $p$ of a special of trivial  character of  $H $ is not 
 supersingular. 
 
 (ii) When $\chi$ extends to $\tilde H=H\otimes \mathbb Z[\Omega]$  we can extend it trivially ($\Omega$ acts by the trivial character).
When $\chi$ does not extend, the normal subgroup   $\Lambda_{Z(G)\tilde Z_0}$ of $\Omega$ has index $2$  \eqref{eq fix}. We  extend $\chi$ trivially to $H\otimes \mathbb Z[\Lambda_{Z(G)\tilde Z_0}]$ and then induce to $\tilde H$. 
 \end{proof}

  \section{Discrete simple   modules  with   supersingular reduction} 

In this section, $\bf G$ is absolutely simple and $P$ is a maximal ideal of $\mathbb Z[q^{1/2}]$ of residue field $\mathbb F_p$.
  
 \begin{proposition}\label{prop key} {\rm [Key result]} There exists a simple discrete right $\tilde H_{\mathbb C} $-module $M_\mathbb C$   with a $\mathbb Z[q^{1/2}]$-integral structure $M$ of supersingular reduction  modulo $P$, except if the type  of $\Sigma$ is  $A_\ell (\ell \geq 1)$ and 
 the parameters  are equal. 
\end{proposition} 
\begin{proof}   A special  complex character of  $ H_{\mathbb C} $ extending  the $\mathfrak B^{sc}$-invariant of the  Steinberg complex representation of $G^{sc} $,   is discrete  and integral but its reduction modulo $p$ is not supersingular. A trivial  complex character is not discrete. 
   The  discrete
    non-special characters complex characters  are integral of supersingular reduction modulo $p$.  The discrete non-special characters   $\chi:H \to \mathbb C$ were computed by Borel \cite[5.8]{Borel}  and those extending to $\tilde H$ have been described in Prop. \ref{prop cha}(i).
Applying   Prop. \ref{prop dis}, we have:

 -   if  a discrete non-special character of $H_{\mathbb C} $ extends to   $\tilde H_{\mathbb C} $, we obtain a simple discrete integral $\tilde H_{\mathbb C} $-module $M_\mathbb C$ of supersingular reduction modulo $p$ and of dimension $1$.  
 
 - if  $ H_{\mathbb C} $ admits  discrete
    non-special characters complex characters but none extends to  $\tilde H_{\mathbb C} $. Using    Prop. \ref{prop cha}(ii), we  obtain a simple discrete integral $\tilde H_{\mathbb C} $-module $M_\mathbb C$ of supersingular reduction modulo $p$ and of dimension $2$. 

 - if the special character is the only discrete character of $ H_{\mathbb C} $,  the reflection left $\tilde H_{\mathbb Z[q^{1/2}]}$-module $M$, free of rank $|S|$ over 
$\mathbb Z[q^{1/2}] $ has a  supersingular    reduction modulo $P$.  The left $\tilde H_{\mathbb C}$-module $M_\mathbb C$ is discrete when the type  is different from $A_{\ell} (\ell \geq 1)$ \cite{Lusztig}.

These are the main lines of the proof.  We give the  details in  the rest of this section.   \end{proof}  
     
 Instead of $\chi=  (\chi_i )$ as in section \S \ref{S 4}, we write $\chi=\chi_{(\epsilon_i)}$ where $\epsilon_i = \begin{cases} -1 & \text{if} \ \chi_i=-1\\  d_i & \text{if} \ \chi_i=q^{d_i}\end{cases}$ for $1\leq i \leq m$.

\begin{proposition} \label{prop 12} There exists a discrete  non special right  $\tilde H_{\mathbb C} $-module $M_\mathbb C$ when the type of $\Sigma$  is   $ B_\ell (\ell \geq 3),  C_\ell (\ell \geq 2),  F_4, G_2$,  $A_1$ with $d_1\neq d_2$.  The $\mathbb C$-dimension $r$ of $M_\mathbb C$ is:

 $r=1$   when $\Psi$ is trivial or when the type is  $ B_\ell (\ell \geq 3),  C_\ell (\ell \geq 6),  F_4, G_2$,  $A_1$ with $d_1\neq d_2$,  
 $C_2$ with parameters $ (1,2,2),(2,3,3)$,   $C_3$ with parameters  $  (2,1,1)$,   $C_4$ with parameters   $( (d,d,d), (2,1,1)$, $C_5$ with parameters   $ (d,d,d), (2,1,1),(2,3,3)$.

 $r=2$ when  the type is $C_2$ with parameters $ (d,d,d), (2,1,1)$,   $C_3$ with parameters  $ (d,d,d)$, $(1,2,2),(2,3,3)$,   $C_4$ with parameters   $ (1,2,2),(2,3,3)$, $C_5$ with parameters   $  (1,2,2)$; then $M_\mathbb C$ extends the $H_{\mathbb C}$-module $\chi_{(-1,-1,d)}\oplus \chi_{(-1,d,-1)}$.\end{proposition} 
\begin{proof}
 When $m=1$ the only   discrete complex character is the special one  \cite[5.7]{Borel}.
When $m>1$   $H_\mathbb C$ admits a discrete non-special  character   except for the type $A_1$ and  parameters $d_1=d_2$ \cite[5.8]{Borel}.

 Assume $m>1$ or the type is $A_1$ and $d_1\neq d_2$. Some discrete non-special  character of $H_\mathbb C$ extends to $\tilde H_\mathbb C $ with the  exception:  type 
$C_\ell (\ell \geq 2) $, different $\epsilon_2\neq \epsilon_3$, $\Psi\simeq \mathbb Z/ 2 \mathbb Z $,   parameters $d_2=d_3$, say $d$,   
 
 -   $d_1= d$ and $\ell=2,3$.

-  $d_1\neq d$ \footnote{As $d_1\neq d$,  the possible parameters  $(d_1,d ,d )$  are 
$( 2,1,1)  (1,2,2),      (2,3,3)  $  \cite[\S 4]{Tits}}
 and  $\ell=2$  with   parameters $( 2,1,1)$,   or 
 $\ell=3, 4$ 
with   parameters $ (1,2,2),   (2,3,3) $, or  
$\ell=5$  with  parameters $ (1,2,2)$.

\noindent In the exceptional case,  no discrete character of $H_\mathbb C$ extends to $\tilde H_\mathbb C $,  $\chi_{ (-1,-1,d)}, \chi_{(-1,d,-1)}$  are discrete characters of  $H_{\mathbb C} $, and the  simple $\tilde H_{\mathbb C} $-module of dimension $2$  equal to  $\chi_{(-1,-1,d)}\oplus \chi_{(-1,d,-1)}$ as a  $H_{\mathbb C} $-module (Prop. \ref{prop cha}) is discrete  by  Prop. \ref{prop dis}.
\end{proof}

 We consider now the types  $D_{\ell} (\ell \geq 2),E_6, E_7, E_8$. Then $\bf G$ is $F$-split and  for distinct $s,t\in S$, the order $n_{s,t}$ of $st$ is $2$ or $3$   \cite[\S4]{Tits}.
 
  \begin{proposition}\label{prop discsinc} 
 Assume that the type of $\Sigma$ is $D_{\ell} (\ell \geq 2),E_6, E_7, E_8$.  Then, the  image $M$ by the automorphism $T_w\mapsto (-1)^{\ell (w)}T_w^*$ of $H (G ,  \mathfrak B )$ of the reflection right $H_{\mathbb Z[q^{1/2}]}(G,  \mathfrak B)$-module  has a supersingular reduction modulo $P$ and $M_\mathbb C$
is a discrete simple right $H_{\mathbb C}(G,  \mathfrak B)$-module  of dimension $|S|$.
\end{proposition} 

\begin{proof}
The  reflection left $H_{\mathbb Z[q^{1/2}] }(G ,  \mathfrak B )$-module  is the free $\mathbb Z[q^{1/2}] $-module   of basis $(e_s)$ for $s\in S$ with the structure  of $H_{\mathbb Z[q^{1/2}] }(G ,  \mathfrak B )$-module satisfying for $s,t\in S$, $u\in \Omega$,
$$T_u(e_t)= e_{u(t)}, \quad T_s (e_t)= \begin{cases} 
- e_t & \text{for }   s=t ,\\
{q}e_t & \text{for }    s\neq t, \ n_{s,t}= 2,\\
{ q}e_t+ q^{1/2}e_s& \text{for   } s\neq t, \ n_{s,t}= 3.
\end{cases}
$$
of  image   by the automorphism $T_w\mapsto (-1)^{\ell (w)}T_w^*$ for $w\in \tilde W$,   satisfying  
$$T_u(e_t)= e_{u(t)}, \quad T_s (e_t)= \begin{cases} 
q e_t & \text{for }   s=t ,\\
- e_t & \text{for }    s\neq t, \ n_{s,t}= 2,\\
- e_t+ q^{1/2}e_s& \text{for   } s\neq t, \ n_{s,t}= 3.
\end{cases}
$$
 The    reduction modulo $P$ of  this left  $H_{\mathbb Z[q^{1/2}] }(G ,  \mathfrak B )$-module, say $M$,  is the $\mathbb F_p$-vector space   of basis $(e_s)$ for $s\in S$ with the structure  of left $H(G,  \mathfrak B)$-module satisfying for $s,t\in S$, $u\in \Omega$,
$$T_u(e_t)= e_{u(t)}, \quad  T_s (e_t)= \begin{cases} 
0& \text{for }   s=t ,\\
-e_t& \text{for }    s\neq t ,
\end{cases}
$$
The restriction to $H_{\mathbb F_p}(G^{sc},  \mathfrak B^{sc})$ of this $H_{\mathbb F_p}(G,  \mathfrak B)$-module  is the direct sum of the   characters $$\chi_s (T_t)=\begin{cases} 0 & \text{for }   s=t ,\\
-1& \text{for }    s\neq t 
\end{cases}\quad \text{ for } s\in S.$$ 
These characters are supersingular hence the  reduction of $M$ modulo $P$  is supersingular (Fact \ref{fact}).  
The  scalar extension of $M$ from $\mathbb Z[q^{1/2}] $ to $\mathbb C$ is  a  discrete simple   $H_{\mathbb C }(G ,  \mathfrak B )$-module $M_\mathbb C$ \cite[4.23]{Lusztig}. 
The properties (discrete scalar extension of $\mathbb C$ and supersingular reduction modulo $P$)  remain true if one sees as $M$ as a right  $H_{\mathbb Z[q^{1/2}] }(G ,  \mathfrak B )$-module  ($T_wv=  vT_{w^{-1}}$ for $w\in \tilde W, v \in M$).  \end{proof}
  
\section{Admissible integral structure via discrete cocompact subgroups}\label{S discrete}

Let $E$ be a number field  of ring of integers $O_E$,  $P_E$ a maximal ideal of $O_E$ with residue field $k=O_E/P_E$, and  $C/E$  a field extension.
\begin{definition}   We say that a  $ C$-representation $\pi$ of $G$ 
   
 a)  descends to $E$  if there exists  an $E$-representation $V$ of $G$  and    a $ G$-equivariant $C$-linear isomorphism
 $\varphi:  C\otimes_{ E}V \to \pi$. We call  $\varphi$ (and more often $V$)    an $E$-structure of $\pi$.

b)    is $O_E$-integral   if $\pi$ contains   a $G$-stable $O_E$-submodule $L$   such that,  for any open compact subgroup $K$ of $G$, the $O_E$-module $L^K$ is finitely generated, and   the natural map
 $\varphi:  C\otimes_{O_E}L \to \pi$   is an isomorphism.
  We  call $\varphi$ (and more often $L$)   an  $O_E$-integral structure of $\pi$ (we say integral if $O_E=\mathbb Z$). The
 $G$-equivariant  map $L\to k\otimes_{O_E}L$ (and more often the $k$-representation $k\otimes_{O_E}L$ of $G$) is called  the reduction of $L$ modulo $P_E$. 
 When $k\otimes_{O_E}L$ is admissible for all $P_E$, we say that $L$ is  admissible.

 We give  analogous definitions for a  $H_{ C}(G,\mathfrak B) $-module.   

 \end{definition}

For any commutative ring $A$ and a discrete cocompact subgroup $\Gamma$ of $G$, let  $\rho^\Gamma_A$ denote the smooth $A$-representation of $G$ acting by right translation on  
$$A(\Gamma \backslash G )^\infty=\{ f :G\to A   \ | \ f(\gamma g k)= f(g) \ (\gamma\in \Gamma, g\in G, k\in K_f \}$$
where $K_f$ is some open compact subgroup of $G$ depending on $f$. The complex representation $\rho_{\mathbb C}^\Gamma$ of $G$ has an admissible integral structure  $\rho^\Gamma:= \rho^\Gamma_\mathbb Z$. The reduction  of  $\rho^\Gamma$ modulo a prime number $c$ is the  admissible
representation $ \rho_{\mathbb F_c}^\Gamma$. The next   proposition supposes $\bf G$ semi-simple only to simplify.
 
  \begin{proposition} \label{prop emb} Assume $a=0$ and $\bf G$ semi-simple. If  $\pi$ is a square integrable  $\mathbb C$-representation of $G$, then there exists a discrete cocompact subgroup $\Gamma$ of $G$ such that $\Hom_G(\pi, \rho^\Gamma_{\mathbb C})\neq 0$.
 \end{proposition}

\begin{proof}  As $a=0$,  there exists a strictly decreasing sequence  $(\Gamma_n)_{n\in \mathbb N}$ of discrete cocompact subgroups of $G$,  normal  of finite index in $\Gamma=\Gamma_0$ \cite{BorelHarder},\cite[IX (4.7) (D) Thm. and  (4.8) Corollary (iv)]{Margulis},   \cite[Prop. 1.3]{Roga}. 
 The normalized multiplicity of $\pi$ in $\rho^\Gamma_{\mathbb C}$ is
$$m_{\Gamma, dg}(\pi):= \vol_\Gamma \dim_{\mathbb C}
(\Hom_ G(\pi, \rho^\Gamma_{\mathbb C}))$$
where $\vol_\Gamma $ is the volume of $\Gamma \backslash G$ for a   $G$-invariant measure induced by a Haar measure on $G$. 
By  the limit multiplicity formula,   the sequence $(m_{\Gamma_n, dg}(\pi))$ converges, and its limit is not $0$ (\cite[Appendice 3]{DKV} plus  \cite{Kazhdan}). 
\end{proof}

\begin{proposition}  \label{prop mod} Assume $a=0$. Let  $\pi$  an irreducible $ \mathbb C$-representation  of $G$ such that   $ \pi^\mathfrak B \neq 0$ and $\Gamma$ a discrete cocompact subgroup of $G$.

1) If  $\varphi:  \mathbb C\otimes_{E}V\to  \pi$  is an $E$-structure of $\pi$, then   $\varphi^\mathfrak B:  \mathbb C\otimes_{E}V^\mathfrak B\to   \pi^\mathfrak B $ is an $E$-structure of  $ \pi^\mathfrak B$, and the natural map $\Hom_{ EG}(V ,\rho^\Gamma_{E}) \to\Hom_{ \mathbb C G}( \pi,\rho^\Gamma_{ \mathbb C}) $  is an isomorphism.

2) If $\psi: \mathbb C\otimes_{E}W\to \pi^\mathfrak B$ is an $E$-structure of  $\pi^\mathfrak B$, then  $\mathfrak T (\psi):\mathbb C\otimes_{E}\mathfrak T(W)\to \pi $  is  an $E$-structure of $\pi$.

 3) An irreducible  subrepresentation $V$ of $\rho ^\Gamma_E$ admits an  admissible  $O_E$-integral structure 
 $V \cap \rho_{O_E}^\Gamma$ of reduction modulo $P_E$ contained in $\rho ^\Gamma_k$.
 \end{proposition}
 \begin{proof}  
 We recall  a general result in algebra \cite[\S 12, n$^o$2 Lemme 1]{BkiA8}: Let $C/C'$ be a field extension, $A$  a $C'$-algebra. For $A$-modules $M,N$,  the natural map 
\begin{equation}\label{eq:K}C\otimes_{C'}\Hom_{A} (M,N)  \to \Hom_{C\otimes_{C'}A} (C\otimes_{C'}M,C\otimes_{C'}N)   \end{equation}
is injective, and bijective if $C/C'$ is finite or the $A$-module $M$ is  finitely generated. 

 1) Take $C/C'=\mathbb C/E$,  $(M,N)= (E[\mathfrak B\backslash G], V)$ or $(V, \rho_E^\Gamma)$. Then 
 \eqref{eq:K} is an isomorphism because $E[\mathfrak B\backslash G]$ (resp. $V$) is an $E$-representation of $G$ generated by the characteristic function of $\mathfrak B$ (resp. irreducible).

2) The functors  $\pi \mapsto \pi^\mathfrak B$-invariants and its left adjoint  $\mathfrak T$ commute with 
 scalar extension  from $E$ to $\mathbb C$ and when $C= \mathbb C$, are inverse equivalences of categories (Fact \ref{fact1}). Hence $\mathfrak T (\psi):
 \mathfrak T( \mathbb C\otimes_{E}W)=\mathbb C\otimes_{E}\mathfrak T(W)\to \mathfrak T(\pi^\mathfrak B)=  \pi $  is an isomorphism.

3) For  any open compact subgroup $K$ of $G$, the $O_E$-module 
  $( \rho_{O_E}^\Gamma)^K $ is  free and $\rho_{O_E}^\Gamma $ contains $L:=  V \cap \rho_{O_E}^\Gamma$ as $O_E$-representations of $G$. The $O_E$-submodule $L^K$ of $( \rho_{O_E}^\Gamma)^K $ is  finitely generated because the ring $O_E$ is noetherean.
  The natural linear $G$-equivariant isomorphism $ E\otimes_{O_E}\rho^\Gamma_{O_E}\to \rho^\Gamma_{E}$  restricts to a linear  $G$-equivariant isomorphism $ E\otimes_{O_E}L\to V$ and therefore $L$ is an $O_E$-integral structure of $V$. We have $P_E \rho_{O_E}^\Gamma\cap L = P_E(\Gamma\backslash G)^\infty\cap L=P_E(\Gamma\backslash G)^\infty\cap V= P_E(O_E(\Gamma\backslash G)^\infty\cap V)=P_EL$. Therefore,  the reduction modulo $P_E$ of $L$, i.e. $(O_E/P_E)\otimes_{O_E} L=k\otimes_{O_E} L$ is contained in the  reduction modulo $P_E$ of $\rho^\Gamma_{O_E}$, i.e. $  \rho^\Gamma_{k}$, as $k$-representations of $G$.    As $  \rho^\Gamma_{k}$ is admissible, $k\otimes L$  is also.  
 \end{proof}     

 Although the ring $\mathbb Z[q^{1/2}]$ is not always the ring of integers $O_E$ of $E= \mathbb Q[q^{1/2}]$,  the  arguments  of 3) apply to this ring. The key proposition Prop. \ref{prop key} implies:
 
  \begin{corollary} \label{cor key} Assume $a=0, {\bf G}$ absolutely simple,   the type of $\Sigma$ is different from $A_{\ell} (\ell \geq 1)$ when the parameters are equal, and 
  $M_\mathbb C$   a   right $H_{\mathbb C}(G, \mathfrak B)$-module as in  Prop. \ref{prop key}. Then the irreducible  square integrable  $\mathbb C$-representation  $\mathfrak T (M_\mathbb C)$  of $G$ admits  an admissible $\mathbb Z[q^{1/2}]$-integral structure.\end{corollary} 
   
When $\bf G$ is semi-simple, an irreducible admissible supercuspidal $\mathbb C$-representation of $G$   descends to a number field \cite[II.4.9]{Viglivre}.  Propositions \ref{prop  emb} and   \ref{prop mod}   imply:

  \begin{corollary} \label{cor sca} Assume $a=0$ and $\bf G$ semi-simple. An irreducible admissible supercuspidal $\mathbb C$-representation  admits an admissible $O_E$-integral structure of reduction modulo $P_E$ contained in $\rho_k^\Gamma$, for some discrete cocompact subgroup $\Gamma$ of $G$.\end{corollary}

   \section{Supercuspidality criterion}\label{S 5}

This section contains the proof of Proposition \ref{prop ss=sc}. We use the  description of the scalar extension of an irreducible admissible representation  $\pi$  of  $G$ (a theorem with Henniart \cite{HenV}).

 The commutant $D$ of $\pi$ is a division algebra of finite dimension over $C$. Let   $E$ denote the center of $D$, $E_s/C$  the maximal separable extension contained in $E/C$ and 
 $\delta$   the reduced degree of $D/E$. Let $C^{alg}$ be an algebraic closure  of  $C$  containing $E$ and  $\pi_{C^{alg}}$  the scalar extension of $\pi$ from $C$ to  $C^{alg}$.

\begin{fact}\label{fact HenV}   \cite{HenV} The length of $\pi_{C^{alg}}$ is $\delta [E:C]$ and 
  $$\pi_{C^{alg}} \simeq \oplus^\delta (\oplus_{i\in \Hom (E_s,C^{alg} )}V_i)$$ where  $V_i$ is indecomposable of commutant  $C^{alg}\otimes_{i,E_s}E$,  descends to a finite extension $C'$ of $C$,   has length  $[E:E_s]$ and its  irreducible subquotients are all isomorphic, say to $\rho_i$. The $\rho_i$ are admissible, of commutant  $C^{alg}$,  
 $\Aut_C(C^{alg})$-conjugate, not isomorphic each other, descend to a finite extension $C'/C$, and the descents seen as  $C$-representations of $G$,  are $\pi$-isotypic of  finite length. 
\end{fact}
It is straightforward to deduce, using that parabolic induction commutes with scalar extension from $C$ to $C^{alg}$, $\Aut_C(C^{alg})$,  descent and finite direct sums, the following compatibility between  supercuspidality and scalar extension from $C$ to $C^{alg}$:

\begin{proposition}\label{prop  HenV}$\pi$ is supercuspidal if and only if some irreducible subquotient $\rho$ of  $\pi_{C^{alg}} $ is supercuspidal if and only if any irreducible subquotient $\rho$ of  $\pi_{C^{alg}} $  is supercuspidal.
\end{proposition}
 
\begin{remark} \label{rem sq} {\rm When $C\subset \mathbb C$,    some irreducible subquotient $\rho$ of  $\pi_{\mathbb C} $ is square integrable modulo center, if and only 
 if any irreducible subquotient $\rho$ of  $\pi_{\mathbb C} $  is square integrable modulo center.} 
\end{remark}

The pro-$p$ Sylow subgroup  $\mathfrak U$ of the Iwahori subgroup $\mathfrak B$ of $G$ is a pro-$p$ Iwahori subgroup of $G$. The theory of  modules over the  pro-$p$ Iwahori Hecke ring $H(G, \mathfrak U)$ is similar to the theory of  the Iwahori Hecke ring $H(G, \mathfrak B)$ and we refer the reader to \cite{Vig1} and  to \cite{Vigss} for the  description of $H(G, \mathfrak U)$ and  of  the  supersingular right $H_C(G, \mathfrak U)$-modules.
It is easy to see that the $\mathfrak U$-invariant functor and supersingularity commute with scalar extension from $C$ to $C^{alg}$, $\Aut_C(C^{alg})$,  descent and finite direct sums. 

\medskip We prove Proposition \ref{prop ss=sc}, by descending  to $C$ the supercuspidal criterion proved over $C^{alg}$ with  Ollivier  in \cite[Thm.3]{OV}. 

Suppose that $\pi$ contains a non-zero $\mathfrak U$-invariant supersingular element. Some irreducible subquotient $\rho$  of  $\pi_{C^{alg}} $ does also.  By  the supercuspidality criterion over $C^{alg}$, $\rho$ is supercuspidal and  the $H_{C^{alg}}(G, \mathfrak U)$-module $\rho^\mathfrak U$ is supersingular.   A descent $\rho'$ of $\rho$ to a finite extension $C'/C$ is also supercuspidal and    $\rho'\,^{\mathfrak U}$    is also a supersingular $H_{C'}(G, \mathfrak U)$-module.  As $\rho'$ is $\pi$-isotypic as a $C$-representation,  $\rho'\,^{\mathfrak U}$ is  a direct  sum of $\pi^\mathfrak U$ as a $H_C(G,\mathfrak U)$-module. Hence  $\pi$ is  supercuspidal and  the $H_{C }(G, \mathfrak U)$-module $\pi^\mathfrak U$ is  supersingular.

Conversely, suppose that   $\pi$ is supercuspidal. By Prop.\ref{prop HenV},  an irreducible subquotient $\rho$  of  $\pi_{C^{alg}} $ is supercuspidal. By the supercuspidality criterion over $C^{alg}$,  the   $H_{C^{alg}}(G, \mathfrak U)$-module $\rho^{\mathfrak U}$ is  supersingular.  As above, we deduce that  the 
$H_C(G,\mathfrak U)$-module $\pi^\mathfrak U$ is supersingular. $\square$

  \section{Change of coefficient field} \label{S Cfinite}\label{S 8}

This section contains the proof  of Prop.\ref{prop Cfinite}, divided in three steps. Let   $F_c= \begin{cases} \mathbb Q \text{ if } c=0\\
\mathbb F_c \text{ if } c\neq 0
\end{cases}$ be the prime field of characteristic $c$ and $C/F_c$  a field  extension of algebraic closure $C^{alg}/F_c^{alg}$. 

 Step  1  shows that, if $c\neq p$ and  $G$ admits  an irreducible admissible  supercuspidal $C$-representation $\pi$, then $G$ admits one over a finite extension of  $F_c$.

 Indeed, twisting by a $C^{alg}$-character of $G $ we can suppose  that the values of the central character of $\pi$ belong to $F_c^{alg} $.  By \cite[II.4.9]{Viglivre}, when 
 $c\neq p$,   $\pi$   descends to a finite extension  $F_c'/F_c$. The descent preserves irreducibility, admissibility and  supercuspidality. A  descent of $\pi$ from $C^{alg}$ to $F_c'$ is an irreducible admissible  supercuspidal $F_c'$-representation of $G$.

\smallskip     Step 2 shows that if $G$ admits an  irreducible admissible  supercuspidal  representation    over a finite extension of  $F_c$ then   $G$ does over $F_c$.
 
 Indeed, let $C/C'$ be a finite field extension and $\pi$ and  irreducible admissible  $C$-representation of $G$. Then $\pi$ is  admissible and finitely generated
 as a $C'$-representation of $G$. This implies that $\pi$  contains an irreducible admissible $C'$-representation $\pi'$  (\cite[Lemma 7.10]{HV2}  for $c=p$ and the next lemma for $c\neq p$). By adjunction, $\pi$ is a quotient of the scalar extension $\pi'_C$ of $\pi'$ from $C'$ to $C$. We prove that if $\pi$ is  supercuspidal then $\pi'$ is  supercuspidal. Suppose that $\pi'$ is not supercuspidal,  subquotient of $\ind_P^G \tau'$ for a proper parabolic subgroup $P$ of $G$ and $\tau'$ an irreducible admissible $C'$-representation of a Levi subgroup $M$ of $P$. The parabolic induction is compatible with the the scalar extension  from $C'$ to $C$, hence  $\pi'_C$ is a  subquotient of $\ind_P^G \tau'_C$ and $\pi$ is also. The $C$-representation $\tau'_C$ of $M$  has finite length and its irreducible subquotients are admissible \cite[Cor.4]{HenV}. Hence $\pi$ is a subquotient of $\ind_P^G \rho$ for some irreducible admissible subquotient of  $\tau'_C$, therefore $\pi$ is not supercuspidal.

\smallskip Step 3 shows that if  $G$ admits an  irreducible admissible  supercuspidal  $F_c$-representation   then $G$ does over any field of characteristic $c$.

Indeed, let $L/C$ be a field extension of  algebraic closure
 $L^{alg}/C^{alg}$ and $\pi$ an  irreducible admissible  supercuspidal  $C$-representation of $G$.  Let  $\tau$ be an  irreducible subquotient of the scalar extension $\pi_L$ of $\pi$  from $C$ to $L$. It is admissible \cite[Cor.4]{HenV}.
 The scalar extension $\tau_{L^{alg
  }}$  of $\tau$ from $L$ to $L^{alg}$ is a subquotient  of  the scalar extension  $\pi_{L^{alg}}$ of $\pi$ from $C$ to $L^{alg}$, equal to the scalar extension of $\pi_L$ from $L$ to $L^{alg}$.
   By \S\ref{S 5},  $\pi_{L^{alg}}$ has finite length, the  irreducible subquotients of $\pi_{L^{alg}}$ are absolutely irreducible, admissible,  supercuspidal and descend to a finite extension  $C'/C$, and the descents are $\pi$-isotypic as $C$-representations of $G$. 
  Therefore $\tau_{L^{alg
  }}$ has the same properties. We deduce that   the extension $\tau_{L'}$ of $\tau$ from $L$ to a finite extension $L'$ of $L$  has finite length, the  irreducible subquotients of $\tau'$ are absolutely irreducible,  admissible, supercuspidal and $\tau$-isotypic as $L$-representations of $G$. This implies that $\tau$ is   supercuspidal.
 
  To end the proof of Prop.\ref{prop Cfinite}, it remains to prove the lemma announced in Step 2:

  \begin{lemma} Assume $c\neq p$. An admissible   finitely generated $C$-representation $\pi$ of $G$ has finite length.
\end{lemma}
\begin{proof} 
The extension $\pi_{C^{alg}}$  of $\pi$ from $C$ to  $C^{alg}$  is also admissible  finitely generated. The lemma  is known when $C$ is algebraically closed \cite[II.5.1]{Viglivre} hence $\pi_{C^{alg}}$ has finite length, implying that $\pi$ has finite length.
\end{proof}

\section{Reduction to an absolutely simple adjoint group }\label{S 2}

As well known, the adjoint group  $\bf G^{ad}$  of $\bf G$ 
 is $F$-isomorphic   to  a  finite product of  reductive connected $F$-groups \begin{equation}\label{eq Gad}
 {\bf G^{ad} }\simeq  {\bf H} \times \prod_i R_{F'_i/F}( {\bf G'_i})
 \end{equation} 
 where  
  $H:={\bf H}(F)$ is compact (hence any  smooth irreducible $C$-representation of $H$ is admissible and supercuspidal), the $F'_i/F$  are  finite separable extensions   
and $R_{F'_i/F}( {\bf G'_i})$ are scalar restrictions from $F'_i$ to $F$  of   absolutely simple adjoint connected    $F'_i$-groups $\bf G'_i $,
 
\begin{proposition} \label{prop redadjoint}  Assume that the field $C$ is  algebraically closed  or   finite. If for any $i$, $G'_i:={\bf G'_i}(F')$ admits  an irreducible admissible supercuspidal $C$-representation, then $G$ 
admits  an irreducible admissible supercuspidal $C$-representation.

\end{proposition}

For $C=\mathbb C$ and   the
 finite group analogue  $\mathbf{H}(k)$ of $G$, where $\bf H$ is a connected reductive  group over a finite field $k$ of characteristic $p$,  this is proved in \cite[Proof of Prop. 2.1]{Kret}. 
 
\medskip  The proposition is the combination of the  Propositions \ref{prop GtimesH}, \ref{prop GHfinite}, \ref{prop HtoG}, \ref{prop RF'toF}  corresponding to the   operations: 
 finite product,  central extension, scalar restriction, with $C$ algebraically closed or finite.  
  
\medskip  1) {\it Finite product}
   Let ${\mathbf   G_1}, {\mathbf   G_2}$ be two connected  reductive $F$-groups and  $\sigma, \tau $  irreducible admissible $C$-representations  of $G_1:={\mathbf   G_1}(F) ,G_2:={\mathbf   G_2}(F)$.
 
 \begin{proposition} \label{prop GtimesH} Assume that $C$ is algebraically closed.  
  
a)  The tensor product  $ \sigma \otimes_C\tau$  is an irreducible admissible $C$-representation of $G_1\times G_2$.
 
b) Each irreducible admissible $C$-representation of $G_1\times G_2$ is of this form.

c)  $ \sigma \otimes_C\tau$ determines $\sigma,\tau $ (modulo isomorphism).

d)  $ \sigma \otimes_C\tau$   is supercuspidal  if and only if   $\sigma$ and $\tau$ are supercuspidal.   \end{proposition}
 
\begin{proof}   $\sigma\otimes \tau$  is admissible as, for  open compact  subgroups $K_i$ of $G_i$, we have a  
  natural isomorphism \cite[\S 12, 2 Lemme 1]{BkiA8}:
 $$ \Hom_{K_1 }(1 , \sigma )\otimes  \Hom_{ K_2}(1 ,  \tau)\to  \Hom_{K_1\times K_2}(1  \times 1, \sigma \otimes \tau).$$
  Suppose now $C$ algebraically closed. 
 
 a) It is known that  $\sigma\otimes_C \tau$ is irreducible \cite[\S 12, 2 Cor.1]{BkiA8} (the commutant of $\sigma$ is $C$  \cite{HenV}).

b) Let  $\pi$ be an  irreducible admissible $C$-representation of $G_1\times G_2$ and let $K_1,K_2$ be any compact open subgroups of $G_1, G_2$ such that $\pi^{K_1\times K_2}\neq 0$.

When $c=p$, the $C$-representation of $G_1$ generated by $\pi^{K_2}$ is admissible  as $\pi^{K_1\times K_2}$ is finite dimensional over $C$. It contains an irreducible admissible $C$-subrepresentation $\sigma$  \cite[Lemma 7.10]{HV2}. Let $\tau:= \Hom_{G_1}(\sigma, \pi)$ with the natural action of $G_2$. The representation $\sigma \otimes_C\tau$ embeds naturally in $\pi$; as $\pi$ is irreducible, it is isomorphic to $\sigma \otimes_C\tau$, and $\tau$ is irreducible. As $\pi$ is admissible, $\tau$ is admissible \footnote{The proof is due to Henniart}.

When $c\neq p$,   $\pi^{K_1\times K_2}$ is a  simple right $H_C(G_1\times G_2, K_1\times K_2)$-module   \cite[I.6.3]{Viglivre}. We have 
$  H_C(G_1\times G_2, K_1\times K_2)\simeq H_C(G_1, K_1)\otimes_C H_C(G_2,K_2)$. By \cite[\S 1 Proposition 2]{BkiA8},  the finite dimensional simple $H_C(G_1, K_1)\otimes H_C(G_2,K_2)$-modules are factorizable  (we can also imitate with the Hecke algebras  the argument above for $c=p$) hence  $\pi^{K_1\times K_2}\simeq \sigma^{K_1} \otimes \tau^{K_2}$  for irreducible admissible $C$-representations 
$\sigma, \tau$ of $G_1,G_2$, and 
 $\pi \simeq \sigma \otimes_C \tau$  \cite[I.6.3]{Viglivre}.

c) As a $C$-representation of $G_1$,  $\sigma\otimes_C\tau$ is $\sigma$-isotypic.
As a representation of   $G_2$, $\sigma\otimes_C\tau$ is $\tau$-isotypic. Hence c). 

d) The parabolic subgroups of $G_1\times G_2$ are product of parabolic subgroups of $G_1$ and of $G_2$. Let  $P,  Q$ be  parabolic subgroups of $G_1,G_2$  of Levi subgroups  $M, L$ and let  $\pi_1$ be an irreducible admissible $C$-representation  of $M\times L$. By b), $\pi_1$ is factorizable, $\pi_1=\sigma_1 \otimes_C \tau_1$ for irreducible admissible $C$-representations $\sigma_1, \tau_1$ of $M,L$; we have a  natural isomorphism $\ind_P^{G_1}\sigma_1 \otimes_C \ind_Q^{G_2}\tau_1\to \ind_{P\times Q}^{G_1\times G_2}\pi_1$. The   irreducible subquotients of $\ind_{P\times Q}^{G_1\times G_2}\pi_1 $ are the tensor products of 
 the irreducible subquotients of $\ind_P ^{G_1} \sigma_1$ and of $\ind_Q^{G_2} \tau_1$. Hence d).
  \end{proof}

Until the end of 1), we assume that $C$ is a finite field.     The next proposition is deduced from \cite[Thm.1]{HenV}.
  \begin{proposition}\label{prop GHfin} 
Let
  $\pi $ be an irreducible admissible $C$-representation  of $G$. The   commutant  $D$ of $\pi$ is   a finite extension of $C$ and the 
  scalar extension $\pi_D$ of $\pi$ from $C$ to  $D$  is  isomorphic to
   $$\pi_D\simeq \oplus _{i\in \Gal(D/C)}\pi_i$$
for   irreducible admissible $ D$-representations $\pi_i$ of $G$ of commutant $D$, not isomorphic to each other,  forming a single  $\Gal(D/R)$-orbit and seen as $C$-representations, are isomorphic to $\pi$.
 \end{proposition}
  \begin{proof} The commutant $D$ of $\pi$ is a  division algebra of finite dimension over $C$, and a finite extension of $C$ is a  finite field and is Galois.  As a $C$-representation, $\pi_D$ is $\pi$-isotypique of length $[D:C]$.  Apply  \cite[Thm.1]{HenV} with  $R'=D$.
  \end{proof}  
    
Returning to $\sigma$ and $\tau$ when $C$ is finite,    the commutants $D_\sigma$ and $D_\tau$ of $\sigma$ and $\tau$ are  finite extensions of $C$ of  dimensions $d_\sigma, d_\tau$. We embed them in  $C^{alg}$ and we consider  the  field $D$ generated by  $D_\sigma$ and $D_\tau$ of $C$-dimension $\lcm(d_\sigma, d_\tau)$, and the intersection $ D'=D_\sigma\cap D_\tau$ of  $C$-dimension $\gcd(d_\sigma, d_\tau)$. The fields  $D_\sigma, D_\tau$ are linearly disjoint on $D'$,
$D_\sigma\otimes_{D'}D_\tau\simeq D$,     
 and  \begin{equation}\label{eq tensor} D_\sigma \otimes_C D_\tau \simeq \prod^{[D':C]} D,\end{equation}
because  $D_\sigma \otimes_C D_\tau \simeq D_\sigma \otimes_{D'}(D'\otimes_CD')\otimes_{D'} D_\tau$,  
  $D'\otimes_CD'\simeq  D' \otimes_C C[X]/(P[X])\simeq D'[X]/(P[X])\simeq  \prod^{[D':C]} D'$  for any $P[X]\in C[X]$  irreducible  of  degree $[D':C]$.   
   
      \begin{proposition}\label{prop GHfinite}   
The $C$-representation $ \sigma\otimes_C \tau$ of $G_1\times G_2$  is isomorphic to
$$ \sigma\otimes_C \tau \ \simeq \ \oplus_{j =1}^{\gcd(d_\sigma, d_\tau)} \ \pi_j$$ 
for   irreducible admissible $C$-representations $\pi_j$ of commutant  $D$ and not isomorphic to each other. We have
  $\sigma$ and $ \tau$   supercuspidal $\Leftrightarrow$ all the $\pi_j$  supercuspidal $\Leftrightarrow$ some $\pi_j$ supercuspidal.
    \end{proposition}

  \begin{proof}   From Prop.\ref{prop GHfin}, the scalar extensions $\sigma_D, \tau_D$ of $\sigma, \tau$ from $C$ to $D$ are isomorphic to 
     $$\sigma_D\simeq \oplus _{i\in \Gal(D_\sigma /C)}\sigma_i,\quad \tau_D\simeq \oplus _{r\in \Gal(D_\tau /C)}\tau_r$$
for   irreducible admissible $ D$-representations $\sigma_i, \tau_r$ of $G_1, G_2$ of commutant $D$, not isomorphic to each other,  forming a single  $\Gal(D/R)$-orbit, descending to $D_\sigma$ (resp. $D_\tau$)  and seen as $C$-representations, isomorphic to $\sigma$ (resp. $\tau$). The $C$-representation $\sigma\otimes_C\tau$  of $G_1\times G_2$ is admissible, of  scalar extension from $C$ to $D$:   \begin{equation}\label{eq tensor1}( \sigma\otimes_C \tau)_D\  \simeq \   \sigma_D\otimes_D \tau_D \ \simeq \ \oplus_{(i,r)\in \Gal (D_\sigma/C)\times \Gal (D_\tau/C)}\  \sigma_i \otimes_D \tau_r.
 \end{equation}
The $D$-representation  $\sigma_i \otimes_D \tau_r$  of $G_1\times G_2$  is admissible of commutant  $D\otimes_D D=D$  \cite[\S12, no 2,  lemma 1]{BkiA8}. Hence   $\sigma_i \otimes_D \tau_r$  is absolutely  irreducible and $( \sigma\otimes_C \tau)_D$ is semi-simple \eqref{eq tensor1}. This implies that  $ \sigma\otimes_C \tau$   is semi-simple \cite[\S12, no7, Prop.8]{BkiA8}; its commutant contains  $D_\sigma \otimes_C D_\tau$. 
From  Prop.\ref{prop GHfin}, $ \sigma\otimes_C \tau$ has  length  $d_\sigma d_\tau/ \lcm(d_\sigma, d_\tau)= \gcd (d_\sigma, d_\tau)$; by \eqref{eq tensor}, its commutant is $D_\sigma \otimes_C D_\tau$ and 
its irreducible  components $\pi_j$ are admissible of commutant $D$ and not isomorphic to each other.

 From Prop. \ref{prop GtimesH} over $C^{alg}$ ,  $\sigma_i\otimes \tau_r$ is supercuspidal if and only if $\sigma_i$ and $\tau_r$ are, if and only if all $\sigma_i$ and all $\tau_r$ are. From Prop. \ref{prop HenV}, this is also equivalent to $\pi_j$ supercuspidal for some $j$, and to $\pi_j$ supercuspidal  for all $j$.      \end{proof}

\medskip 2) {\it  Central extension} 
The natural surjective $F$-morphism ${\mathbf   G} \xrightarrow{\bf i} {\mathbf   G^{ad}}$ of kernel  $\mathbf  { Z(G)}$ induces between  the $F$-points an exact sequence 
$$0\to Z(G)\to  G \xrightarrow{i} G^{ad}\to H^1(F, \mathbf{   Z(G)} ).$$
The image $i(G)$ of $G$ is a closed cocompact normal subgroup of $G^{ad}$ and $H^1(F, \mathbf   {Z(G)} )$ is commutative.  

Until the end of 2), we assume $a=0$. The group $H^1(F, \mathbf   {Z(G)} )$ is  finite  \cite[Thm. 6.14]{PlaRap} implying that   
 $i(G)$ is an open  normal subgroup of $G^{ad}$ of finite commutative quotient.

\begin{proposition} \label{prop HtoG}  $G^{ad}$ admits an irreducible admissible supercuspidal $C$-representation if and only if $G$ does.\end{proposition}

 The inflation from $i(G)$ to $G$  identifies the representations of $i(G)$   with the representations of $G$ trivial on $Z(G)$; it respects irreducibility and admissibility.
The functor (inflation from $i(G)$ to $G$) $\circ$ (restriction from $G^{ad}$ to $i(G)$) from $C$-representations of $G^{ad}$ to representations  of $G$ trivial on $Z(G)$ is denoted by $ - \circ i$.

Let $\tilde \rho$ be  an irreducible  admissible $C$-representation  of $G$ inflating a representation  $\rho$  of  the normal open subgroup $i(G)$ of finite index in $G^{ad}$. The $C$-representation  $\rho$ of $i(G)$ is irreducible  admissible and induces a  representation $\ind_{i(G)}^{G^{ad}}\rho$  of $G^{ad}$ which   is admissible of finite length. Any irreducible quotient  $\pi$ of  $\ind_{i(G)}^{G^{ad}}\rho$  is admissible  (when $c=p$ this uses that  $a=0$), by adjunction $\pi|_{i(G)}$ contains a subrepresentation isomorphic to $\rho$ and by inflation from $i(G)$ to $G$,  $\tilde \rho$ is isomorphic to a subquotient of $\pi \circ i$.
  
\medskip Let $\pi$ be an irreducible admissible    $C$-representation  of $G^{ad}$. The restriction $\pi|_{i(G)}$ of $\pi$ to  $i(G)$  is  semi-simple   of finite length,   its irreducible components $\rho$ are $G^{ad}$-conjugate \cite[I.6.12]{Viglivre}\footnote{ in \cite[I.6.12]{Viglivre} the condition that  the index is invertible in $C$ is not necessary and not used in the proof}. 
The  $C$-representation  $\pi \circ i$ of $G$  is  semi-simple   of finite length, of irreducible components  the inflations $\tilde \rho$  of the irreducible components $\rho$ of $\pi|_{i(G)}$. 
Remark that if   $\pi'$  is a  smooth  $C$-representation  of $G^{ad}$ such that some subquotient  of $\pi'\circ i $ is isomorphic to some $\tilde \rho$, then some subquotient of $\pi'\circ i $   is isomorphic to $ \pi \circ i$.

\medskip Proposition \ref{prop HtoG} follows from:

\begin{proposition}\label{prop HG} $\pi$ is supercuspidal if and only if some $\tilde \rho$ is supercuspidal if and only if  all $\tilde \rho$ are supercuspidal.
\end{proposition}
\begin{proof} It suffices to  prove: $\pi$ not supercuspidal $\Rightarrow $ all $\tilde \rho$ not supercuspidal, and then,  some $\tilde \rho$ not supercuspidal $\Rightarrow $ $\pi$ not supercuspidal. We check first the compatibility of the parabolic induction with $- \circ i$.

(i) The parabolic $F$-subgroups of $\bf G$ and of $\bf G^{ad}$ are in bijection. If the parabolic $F$-subgroup $\bf P$ of $\bf G$ corresponds to the parabolic $F$-subgroup $\bf Q$ of  $\bf G^{ad}$, $\bf i$ restricts to an isomorphism between their unipotent radicals   \cite[22.6 Thm.]{BorelLAG},  sends a Levi subgroup $\bf M$ of $\bf P$ onto a Levi subgroup $\bf L$ of $\bf Q$, and induces between the $F$-points  the exact sequence: 
$$0\to Z(G)\to M\xrightarrow{i}L\to H^1(F, {\mathbf   Z(G)} ).$$  We have $G^{ad}=Q i(G)$, $Q\cap i(G)=i(P)=i(M)U $ where $i(M)$  is an open  normal subgroup of $L$  of finite commutative quotient and $U$ is the unipotent radical of $Q$. If $\sigma$ is a smooth   $C$-representation of $L$, then   $(\ind_Q^{G^{ad}}\sigma)|_{i(G)}\simeq \ind_{i(P)}^{i(G)}(\sigma|_{i(M)})$ and by inflation from $i(G)$ to $G$:
    \begin{equation}\label{eq ind}(\ind_Q^{G^{ad}}\sigma) \circ i\simeq \ind_P^G (\sigma\circ i).\end{equation}

(ii) Let  $\pi$ be an irreducible admissible  not supercuspidal $C$-representation of $G^{ad}$, isomorphic to a subquotient  of  $\ind_Q^{G^{ad}} \sigma$ for  $Q\neq G^{ad}$ and $\sigma$ irreducible admissible $C$-representation of $L$.   Therefore  $\pi\circ i$ is isomorphic to a subquotient  of  $(\ind_Q^{G^{ad}} \sigma) \circ i$, and by \eqref{eq ind} all the  $\tilde \rho$ are isomorphic to a subquotient  of $\ind_P^G \tilde \tau $  for some irreducible subquotient $\tilde \tau$ of $ \sigma \circ i$ (depending on $\rho$). 
As $\tilde \tau$ is admissible and $P\neq G$,  all the   $\tilde \rho$ are not supercuspidal.

(iii) Let $\pi$ be an irreducible admissible $C$-representation of $G^{ad} $ such that  some irreducible component $\tilde \rho$ of $\pi\circ i$ is not supercuspidal,  isomorphic to a subquotient of $\ind_P^G\tau$ for $P\neq G$  and  $\tau'$ irreducible  admissible $C$-representation of $M$. 
The central subgroup $Z(G)$ acts trivially on $\tilde \rho$ hence also on $\tau'$. Therefore $\tau'= \tilde \tau$ for an irreducible subquotient  $\tau$ of $\sigma|_{i(G)}$ where  $\sigma$  is an irreducible admissible $C$-representation of $L$. The representation $\tilde \rho $  is isomorphic to a subquotient  of $\ind_P^{G}(\sigma \circ i)$. By \eqref{eq ind} and the  remark above Prop. \ref{prop HG}, $\pi \circ i$ is isomorphic to a subquotient of $(\ind_Q^{G^{ad}}\sigma) \circ i$. By adjunction  $\pi$  is isomorphic to a subquotient of  $\ind_{i(G)}^{G^{ad}}((\ind_Q^{G^{ad}}\sigma)|_{i(G)})$. This representation is   isomorphic to
    \begin{align*}\ind_{ i(M)U}^{G^{ad}}(\sigma|_{i(M)})\simeq\ind_Q^{G^{ad}}(\ind_{ i(M)}^{L}(\sigma|_{i(M)})) \simeq\ind_Q^{G^{ad}} (\sigma \otimes_C C[i(M)\backslash L]).
 \end{align*}
 The $C$-representation $\sigma\otimes_C C[i(M)\backslash L]$ of $L$ has finite length and its irreducible subquotients $\nu$ are admissible  (Lemma \ref{lem te} below). Therefore  $\pi$ is isomorphic to a subquotient  of  $\ind_Q^{G^{ad}} \nu$ for some  $\nu$ and $Q\neq G^{ad}$,   hence $\pi$ is not supercuspidal.
   \end{proof}

  \begin{lemma} \label{lem te} Let $\pi$ be an irreducible admissible $C$-representation of $G$  and let $V$ be a finite dimensional smooth  $C$-representation of $G$. 
  Then the representation $\pi \otimes_C V$ of $G$ has finite length and its irreducible subquotients are admissible.
  \end{lemma}
  \begin{proof} The scalar extension  $(\pi\otimes_C V)_{C^{alg}}$ of  $\pi\otimes_C V$ to an algebraic closure $C^{alg}$ of $C$ is isomorphic to $\pi_{C^{alg}}\otimes_C V_{C^{alg}}$.  The length of  $\pi \otimes_C V$ is bounded above by the length of $(\pi\otimes_C V)_{C^{alg}}$ and  a $C$-representation of $G$  is admissible if and only if its scalar extension from $C$ to $C^{alg}$ is admissible. By Fact \ref{fact}, 
  the  length of the $C^{alg}$-representation  $\pi_{C^{alg}}\otimes_C V_{C^{alg}}$ of $G$  is the product of the finite lengths of $\pi_{C^{alg}}$ and of $V_{C^{alg}}$, and its irreducible subquotients are $\pi_i \otimes \chi$ for the irreducible quotients $\pi_i$ of $\pi_{C^{alg}}$ and the $C^{alg}$-characters $\chi$ of $L$ trivial on $i(M)$. As $\pi_i$ is admissible, $\pi_i \otimes \chi$ is admissible, and all subquotients of $\pi_{C^{alg}}\otimes_C V_{C^{alg}}$ are admissible.
     \end{proof}

\bigskip 3) {\it Scalar restriction} Let $F'/F$ be a finite separable extension, $\bf G'$ a connected reductive $F'$-group and   $\bf G:=R_{F'/F}(\bf G')$ the scalar restriction of $\bf G'$ from $F'$ to $F$.  As topological groups, $G':={ \bf G'}(F')$  is equal to $G:={\bf G}(F)$.  By \cite[6.19. Cor.]{BorelTits}, $G'$ and $G$ have the same parabolic subgroups, hence:

\begin{proposition}\label{prop RF'toF}     $G' $ admits an irreducible admissible supercuspidal $C$-representation if and only if $G $ does.
 \end{proposition}

\bibliographystyle{amsalpha}

\bibliography{bib}

% \bigskip MSC 2010:  primary 20C08, secondary  11F70
% Keywords: change of weight, Satake transform, compact induction, parabolic induction, pro-$p$ Iwahori Hecke algebra.
\end{document}